\documentclass[12pt,oneside,reqno]{amsart}
\usepackage[latin1]{inputenc}
\usepackage[english]{babel}
\usepackage{amsmath}
\usepackage{amsthm}
\usepackage{amssymb}
\usepackage{amsfonts}
\usepackage{enumitem}
\usepackage{mathrsfs} 
\usepackage{mathtools}
\usepackage{hyperref}
\usepackage[paper=a4paper]{geometry}
\usepackage[square,numbers]{natbib}
\author{Henri Elad Altman}
\address{Freie Universit\"{a}t Berlin, Germany}
\email{henrialtman@mi.fu-berlin.de}
\title[Taylor Estimates for Pinned Bessel Bridges and IbPF]{Taylor Estimates for the laws of pinned Bessel bridges, and Integration by Parts}
\date{\today}

\renewcommand{\d}{\, \mathrm{d}}
\theoremstyle{plain}
\newtheorem{thm}{Theorem}[section]
\newtheorem{lm}[thm]{Lemma}
\newtheorem{prop}[thm]{Proposition}

\theoremstyle{definition}
\newtheorem{df}[thm]{Definition}
\newtheorem{nota}[thm]{Notation}
\newtheorem{rk}[thm]{Remark}
\numberwithin{equation}{section}
\begin{document}
\maketitle


\begin{abstract}
In this article, we extend the integration by parts formulae for the laws of Bessel bridges obtained in \cite{EladAltman2019}, by showing that these formulae hold for very general test functionals on $L^{2}(0,1)$. A key step consists in establishing new Taylor estimates on the laws of pinned Bessel bridges.
\end{abstract}

\section{Introduction}

Integration by parts is one of the most fundamental tools in analysis, and most notably in stochastic analysis. In particular, it plays a central role in the study of some stochastic PDEs for which the usual tools from stochastic calculus break down. For instance, it lies at the core of the theory of Dirichlet forms (see, e.g.,\cite{fukushima2010dirichlet} for an introduction to the theory). Obtaining integration by parts formulae (IbPF for short) for measures in infinite-dimensional spaces is a difficult task in general. One of the most celebrated examples is the case of the Wiener measure, or variants thereof, like the Brownian bridge on a fixed interval.    

In a series of papers of the 2000s, Zambotti introduced a family of stochastic PDEs with reflection (\cite{zambotti2002integration}) and repulsion from $0$ (\cite{zambotti2003integration}) whose invariant probability measures correspond to the laws, on $L^{2}(0,1)$, of Bessel bridges of dimension $\delta \geq 3$ on $[0,1]$. The solutions to such SPDEs are constructed using the techniques introduced by Nualart and Pardoux in \cite{nualart1992white}. They evolve in the set of nonnegative continuous functions on $[0,1]$, and have been proved to display a rich behaviour which, although reminiscent of Bessel processes, is more subtle and intriguing (see \cite{zambotti2017random}, Chapters 5 and 7, for an overview on their remarkable properties). Moreover, they  
arise naturally to describe the fluctuations of a $\nabla \Phi$ interface evolving on a wall (see \cite{funaki2001fluctuations} and \cite{zambotti2004fluctuations}). 

The key tools to investigate these equations are the IbPF on the laws of Bessel bridges of dimension $\delta \geq 3$ between $0$ and $0$, $P^{\delta}$, established by Zambotti in \cite{zambotti2002integration}, \cite{zambotti2003integration}. The strategy he used to obtain these formulae was essentially to transfer IbPF from laws for which such formulae are already known. More precisely, he proceeded in two steps:
\begin{itemize}
\item interpreting $P^{3}$ as the law of a Brownian bridge conditioned to remain nonnegative, he derived an IbPF on it from the well-known IbPF for the Brownian bridge,
\item exploiting absolute continuity relations between $P^{3}$ and $P^{\delta}$, he transferred the IbPF from the former to the latter.
\end{itemize}

This strategy, which worked very well for $\delta \geq 3$, does not carry over to treat the case $\delta<3$. The main reason is that, in such a regime, the measure $P^{\delta}$ no longer satisfies log-concavity, a very important property when deriving IbPF and studying the associated dynamics \cite{ASZ}. In  addition, when, $\delta < 2$, the absolute continuity relation used in the second step outlined above breaks down. In fact, for several years, Zambotti's results could not be extended to treat Bessel bridges of dimension smaller than $3$. An exception was the case $\delta=1$, corresponding to the law of the reflected Brownian bridge on $[0,1]$. Thus, in \cite{zambotti2005integration}, Zambotti obtained an IbPF for the reflected Brownian motion on $[0,1]$, and in \cite{grothaus2016integration}, Grothaus and Vosshall proved a similar formula for the bridge (also providing a more direct formulation based on sophisticated white noise analysis). In these papers, the proofs relied on the underlying Gaussian structure, and on explicit computations using the Cameron-Martin formula. Such features cannot be exploited for generic values of $\delta$. 

The more recent paper \cite{EladAltman2019} derived IbPF on the laws $P^{\delta}$ of Bessel bridges of all dimension $\delta \in (0,3)$, for a specific class of functionals. More precisely, the authors considered the vector space $\mathcal{S}$ generated by all functionals on 
$C([0,1])$ of the form
\begin{align}
\label{exp_functional_continuous_functions}
\begin{cases}
C([0,1]) \to \mathbb{R} \\
X \mapsto \exp \left( - \langle m, X^{2} \rangle \right),
\end{cases}
\end{align}
where $m$ is a finite Borel measure on $[0,1]$, and 
$$\langle m, X^{2} \rangle := \int_0^1 X(r)^2 \, m(d r). $$
Note that functionals of the form \eqref{exp_functional_continuous_functions} play a special role. Indeed, as a consequence of the remarkable additivity property of squared Bessel bridges (see \cite{shiga1973bessel}), such functionals act on the laws of Bessel processes as a Girsanov transformation corresponding to a deternimistic time-change (see Lemma 3.3 in \cite{EladAltman2019}), thus allowing useful explicit computations. Thus, functionals of the type \eqref{exp_functional_continuous_functions} play, in this context, the same role as functionals of the form $ \exp \left( \langle k, X \rangle \right)$, $k \in C([0,1])$, in the papers \cite{zambotti2005integration} and \cite{grothaus2016integration}. These IbPF allowed to identify the structure of SPDEs that should admit the probability measure $P^\delta$, for $\delta <3$, as reversible measure: these equations appear as singular SPDEs of a new kind, in which the drift term involves Taylor remainders at 0 of the local times of the solution (see (1.11)-(1.13) in \cite{EladAltman2019}). 
The IbPF were also exploited to construct weak stationary solutions of these SPDEs in the special cases $\delta=1,2$, using Dirichlet form techniques, see \cite[Section 5]{EladAltman2019} and \cite[Section 4]{eladaltman2019bessel}. 

%
We recall the IbPF obtained in \cite{EladAltman2019}. For $\delta >0$, $b\geq 0$ and $r\in(0,1)$, let $\Sigma^{\delta}_{r}({\rm d}X \,|\, b)$ denote the finite measure on $C([0,1])$ given by
\begin{equation}\label{Sigma}
\Sigma^{\delta}_{r}({\rm d}X \,|\, b) := \frac{p^{\delta}_{r}(b)}{b^{\delta-1}} \,
 E^{\delta} [ {\rm d}X \, | \, X_{r} = b].
\end{equation}
In the above, $E^{\delta}$ denotes the expectation operator corresponding to the probability measure  $P^{\delta}$ on $C([0,1])$, while, for all $r \in (0,1)$, $p^{\delta}_{r}$ denotes the density of the law of $X_{r}$ under $P^{\delta}$. Recall that, for all $b>0$,

\begin{equation}
\label{density}
p^{\delta}_{r}(b) =  \frac{ b^{\delta -1} }{2^{\frac{\delta}{2} -1} (r(1-r))^{\delta/2} \Gamma(\frac{\delta}{2})} \exp \left(- \frac{b^{2}}{2r(1-r)} \right),
\end{equation}
see Chap. XI.3, in \cite{revuz2013continuous}. The measure $\Sigma^{\delta}_{r}(\,\cdot \,|\, b)$ is meant to be the Revuz measure of the diffusion local time of $(u(t,r))_{t\geq 0}$ at level $b\geq 0$, where $u(t, \cdot)_{t \geq 0}$ is the  infinite-dimensional diffusion with invariant measure $P^{\delta}$ (the existence of such a diffusion is merely hypothetical for $\delta \in (0,3)\setminus \{1,2\}$). We introduce a convenient notation: for any sufficiently differentiable function $f: \mathbb{R}_{+} \to \mathbb{R}$, for all $n \in \mathbb{Z}$, and all $b \geq 0$, we set
\[ \mathcal{T}^{\,n}_{b} f := f(b) - \sum_{0\leq j\leq n} \frac{b^{j}}{j!} \, f^{(j)}(0). \]
In words, for all $b \geq 0$, if $n\geq 0$ then $\mathcal{T}^{\,n}_{b} f$ is the Taylor remainder centered at $0$, of order $n+1$, of the function $f$, evaluated at $b$; if $n<0$ then $\mathcal{T}^{\,n}_{b} f$ is simply the value of $f$ at $b$. Finally, defining for all $\delta > 0$
\[ \kappa(\delta) := \frac{(\delta-1)(\delta-3)}{4}, \]
and setting 
\[ k:= \left\lceil \frac{\delta-3}{2} \right\rceil\geq -1, \]
the IbPF obtained in \cite{EladAltman2019} can be written as follows. For all $\delta \in (0,3)\setminus{1}$, $\Phi \in \mathcal{S}$ and $h \in C^{2}_{c}(0,1)$, it holds 
  
%

\begin{equation}\label{onetothree}
\begin{split}
& E^{\delta} (\partial_{h} \Phi (X) ) + E^{\delta} (\langle h '' , X \rangle \, \Phi(X) )= \\ 
& -\kappa(\delta)\int_{0}^{1}  
 h_{r}  \int_0^\infty b^{\delta-4} \Big[ \mathcal{T}^{\,-2k}_{b} \, \Sigma^{\delta}_{r}(\Phi (X) \,|\, \cdot\,) \Big]
  \d b \d r,
\end{split}
\end{equation}
see Theorem 4.1 in \cite{EladAltman2019}. Recall also that the term
$\mathcal{T}^{\,-2k}_{b} \, \Sigma^{\delta}_{r}(\Phi (X) \,|\, \cdot\,)$
appearing in the formulae is actually the Taylor remainder, at order $k$, centered at $0$, of a smooth function of $b^{2}$. In particular, it is of order $b^{2(k+1)}$ as $b \to 0$, which ensures the integral to be convergent. In Theorem 4.1 of \cite{EladAltman2019}, the authors also obtained the following formula for the critical case $\delta=1$:


\begin{equation}
\label{one}
\begin{split}
& E^{1} (\partial_{h} \Phi (X) ) + E^{1} (\langle h '' , X \rangle \, \Phi(X) ) = \frac{1}{4} \int_{0}^{1} \d r \, h_r\, \frac{{\rm d}^{2}}{{\rm d} b^{2}} \, \Sigma^{1}_{r} [\Phi(X) \, | \, b]  \biggr\rvert_{b=0}  .
\end{split}
\end{equation}
We stress that in both of the above propositions, the test functionals are assumed to lie in the space $\mathcal{S}$. Indeed, this assumption allows to perform explicit computations leading to these formulae, and ensures all the quantities (derivatives and integrals) involved to be well-defined. On the other hand, in \cite{eladaltman2020}, similar IbPF were derived for linear functionals: for such functionals, one obtains exactly the same formulae, but through a very different route, based on exact computations involving hypergeometric functions. It is natural to ask whether the IbPF hold for a wide, more general class of functionals. 
For instance, in \cite{zambotti2002integration} and \cite{zambotti2003integration}, the IbPF on $P^{\delta}$ for $\delta \geq 3$ are established for any bounded, continuous functional on $L^{2}(0,1)$, and one could reasonably expect the IbPF for Bessel bridges of dimension $\delta<3$ shown above to hold with the same generality. 

Achieving such an extension is the goal of the present article. 
To do so, it will be convenient to work with the Hilbert space $L^{2}(0,1)$ rather than the Banach space $C([0,1])$. Therefore, as in \cite[Section 5.2]{EladAltman2019}, we consider the vector space $\mathscr{S}$ generated by functionals on $L^{2}(0,1)$ of the form
\begin{align}
\label{exp_functional}
\begin{cases}
L^{2}(0,1) \to \mathbb{R} \\
X \mapsto \exp \left( - \langle \theta, X^{2} \rangle \right),
\end{cases}
\end{align}
where $\theta : [0,1] \to \mathbb{R}_{+} $ is Borel and bounded. Note that any element of $\mathscr{S}$ defines an element of $\mathcal{S}$ by restriction on $C([0,1])$, so the IbPF above hold for elements of $\mathscr{S}$. We  also introduce the following definition: 

\begin{df}
\label{def_c1b}
For any Banach space $(B,\|\cdot\|)$, let $C^{1}_{b}(B)$ be the space of all $\Phi : B \to \mathbb{R}$ which are bounded, $C^{1}$, with bounded Fr\'echet differential. 
Moreover, let $C^{1,1}_{b}(B)$ be the set of all $\Psi \in C^{1}_{b}(B)$ such that there exists $L > 0$ satisfying
\begin{equation}
\label{lipschitz_differential}
 \forall Z, Z' \in L^{1}(0,1), \quad ||| D\Psi(Z) - D\Psi(Z') ||| \leq L \|Z - Z'\|.
\end{equation} 
where $|||\cdot|||$ denotes the operator norm on $B'$, the topological dual of $B$.
\end{df} 

\begin{rk}
\label{bound_double_order_incr}
For any $\Psi \in C^{1,1}_{b}(B)$ and $L$ as in \eqref{lipschitz_differential}, we have
\[ \forall x, y, z \in B, \quad \left| \Psi(x+y+z) - \Psi(x+y) - \Psi(x+z) + \Psi(x) \right| \leq L \|y\| \ \|z\|. \]
\end{rk}
To extend the IbPF from $\mathscr{S}$ to a larger space of functionals, we use an approximation argument consisting of two ingredients:
\begin{enumerate}
\item the space $\mathscr{S}$ is dense in the space of functionals we are considering, for a certain topology to be specified, 
\item the terms appearing in our formulae are all continuous w.r.t. the above mentioned topology.
\end{enumerate}
Note that, in addressing the second point, we obtain rather strong estimates on the Taylor remainders at $0$ of the laws of pinned Bessel bridges, which are interesting in their own right. The major ingredient in deriving these estimates is the additivity property of squared Bessel processes observed in \cite{shiga1973bessel}, and which extends to squared Bessel bridges \cite{pitman1982decomposition}. Along the way we prove the remarkable fact that, for any $\Phi \in C^{1}_{b}(L^{2}(0,1))$ and $r \in (0,1)$, we have
\[  \frac{d}{db} E^{\delta} [\Phi(X) | X_{r} = b]  \biggr\rvert_{b=0} = 0,\]
see Proposition \ref{bound_lip_func}. This remarkable vanishing property was already observed in \cite{EladAltman2019} and \cite{eladaltman2020}, when $\Phi \in \mathcal{S}$ and when $\Phi$ is a linear functional, respectively. The fact that it remains true for general elements $\Phi \in C^{1}_{b}(L^{2}(0,1))$ supports the conjecture, formulated in \cite{EladAltman2019}, that the local times of the solution to Bessel SPDEs should have a vanishing derivative at $0$, see (1.10) in \cite{EladAltman2019}. 

Combining the two points highlighted above, we are able to prove the following results:

\begin{thm}
\label{general_IbPF_13}
Let $\delta \in (1,3)$. Then, for all $\Phi \in C^{1}_{b}( L^{2}(0,1))$  and all $h \in C^{2}_{c}(0,1)$, we have
\begin{equation}
\label{general_formula_13}
\begin{split}
& E^{\delta} (\partial_{h} \Phi (X) ) + E^{\delta} (\langle h '' , X \rangle \, \Phi(X) )= \\ 
& -\kappa(\delta)\int_{0}^{1}  
 h_{r}  \int_0^\infty b^{\delta-4} \Big[ \mathcal{T}^{\,0}_{b} \, \Sigma^{\delta}_{r}(\Phi (X) \,|\, \cdot\,) \Big]
  \d b \d r,
\end{split}
\end{equation}
\end{thm}
Thus, for Bessel bridges of dimension strictly between $1$ and $3$, the formulae hold for any functional in $C^{1}_{b}(L^{2}(0,1))$. In lower dimensions, we can also generalize the formulae, but to some space distinct from $C^{1}_{b}(L^{2}(0,1))$. Indeed, in order to work, our arguments require some additional regularity on our functionals, this however might just be a technical limitation (see Remark \ref{open_quest_estimates} below). 

\begin{df}
Let $\mathcal{S}C^{1}_{b}\left( L^{1}(0,1) \right)$ be the set of functionals on $L^{2}(0,1)$ of the form 
\begin{equation}
\label{phi_and_psi}
 \Phi(X) = \Psi(X^{2}), \quad X \in L^{2}(0,1),
\end{equation}
where $\Psi \in C^{1}_{b} \left( L^{1}(0,1) \right)$.
\end{df} 
The following remarkable property holds:
\begin{prop}
\label{statement_diff_prop}
Let $\delta \geq 0$ and $\Phi \in \mathcal{S}C^{1}_{b}\left( L^{1}(0,1) \right)$. Then, for all $r \in (0,1)$, the function 
\begin{align*}
\begin{cases} 
\mathbb{R}_{+} \to \mathbb{R} \\
b \mapsto  E^{\delta} [\Phi(X) | X_{r} = b] 
\end{cases}
\end{align*}
is twice differentiable at 0.
\end{prop}
For $\delta=1$, the IbPF can be extended to all elements of $\mathcal{S}C^{1}_{b}\left( L^{1}(0,1) \right)$:
\begin{thm}
\label{general_IbPF_1}
For all $\Phi \in \mathcal{S}C^{1}_{b}\left( L^{1}(0,1) \right)$ and $h \in C^{2}_{c}(0,1)$, we have
\begin{equation}
\label{general_formula_1}
\begin{split}
E^{1} (\partial_{h} \Phi (X) ) + E^{1} (\langle h '' , X \rangle \, \Phi(X) ) = \frac{1}{4} \int_{0}^{1} \d r \, h_r\, \frac{{\rm d}^{2}}{{\rm d} b^{2}} \, \Sigma^{1}_{r} [\Phi(X) \, | \, b]  \biggr\rvert_{b=0}  .
\end{split}
\end{equation}
\end{thm}
Finally, to state the result for $\delta \in (0,1)$, we need to introduce a more particular space:
\begin{df}
\label{def_flip}
Let  $\mathcal{S}C^{1,1}_{b}\left( L^{1}(0,1) \right)$ be the set of functionals on $L^{2}(0,1)$ of the form 
\begin{equation}
\label{phi_psi_lip}
 \Phi(X) = \Psi(X^{2}), \quad X \in L^{2}(0,1),
\end{equation}
where $\Psi \in C^{1,1}_{b}(L^{1}(0,1))$.
\end{df}
We prove the following result:
\begin{thm}
\label{general_IbPF_01}
Let  $\delta \in (0,1)$. Then, for all $\Phi \in \mathcal{S}C^{1,1}_{b}\left( L^{1}(0,1) \right)$ and $h \in C^{2}_{c}(0,1)$, we have
\begin{equation}
\label{general_formula_01}
\begin{split}
& E^{\delta} (\partial_{h} \Phi (X) ) + E^{\delta} (\langle h '' , X \rangle \, \Phi(X) )= \\ 
& -\kappa(\delta)\int_{0}^{1}  
 h_{r}  \int_0^\infty b^{\delta-4} \Big[ \mathcal{T}^{\,-2}_{b} \, \Sigma^{\delta}_{r}(\Phi (X) \,|\, \cdot\,) \Big]
  \d b \d r.
\end{split}
\end{equation}
\end{thm}

The paper is organised as follows: after introducing the notations and stating some useful facts on the laws of squared Bessel bridges in Section \ref{section_notations_and_facts}, we prove density results for $\mathscr{S}$ in large spaces of functionals in Section \ref{section_density}, and establish  Taylor estimates for the laws of pinned Bessel bridges in Section \ref{section_estimates}. Putting these results together, we proceed in Section \ref{section_ibpf} to the proofs of Theorems \ref{general_IbPF_13}, \ref{general_IbPF_1} and \ref{general_IbPF_01}.

{\bf Acknowledgements.}
I am very grateful to Lorenzo Zambotti for innumerable precious discussions. I would also like to thank Cyril Labb\'{e} and Martin Hairer for a useful conversation on the subject. Work on this paper started as I was completing my PhD at the LPSM (UMR 8001) in Sorbonne Universit\'{e} in Paris.

\section{Notations and basic facts}
\label{section_notations_and_facts}

\subsection{Notations}

We need to introduce notations for the various norms and vector spaces that we will consider. 

\begin{nota}
In the sequel, for all $p \in \{1, 2, \infty\}$, we denote by $\|\cdot\|_{p}$ the $L^{p}$ norm on $L^{p}(0,1)$. In the special case $p=2$, we will simply write $\|\cdot\|$ for $\|\cdot\|_{2}$, and denote by $\langle \cdot, \cdot \rangle$ the corresponding inner product.
Moreover, for all $p \in \{1, 2\}$ and all functional $\Phi : L^{p}(0,1) \to \mathbb{R}$, we set:
\[ \| \Phi \|_{\infty} := \sup_{X \in L^{p}(0,1)}{| \Phi(X) |}. \]
Furthermore, for all $\Phi \in C^{1}_{b}(L^{p}(0,1))$, we set
\[ ||| D \Phi |||_{\infty} := \sup_{X \in L^{p}(0,1)}{||| D \Phi(X) |||}, \] 
where $||| \cdot |||$ denotes the norm on $\left( L^{p}(0,1) \right)'$, and  we set
\[ \| \Phi \|_{C^{1}} :=  \| \Phi \|_{\infty} +  ||| D \Phi |||_{\infty} . \]
\end{nota}

We also introduce the following shorthand notations:

\begin{nota}
For all $p \in [1, \infty]$, we denote by $L^{p}_{+}(0,1)$ the subset of nonnegative function in $L^{p}(0,1)$. Moreover, we use the shorthand notation $C([0,1]) := C([0,1], \mathbb{R})$ as well as $C_{+}([0,1]) := C([0,1], \mathbb{R}_{+})$
\end{nota}

\subsection{Squared Bessel bridges and Bessel bridges}

In the sequel, for all $\delta, x, y \geq 0$ ,  we will denote by $Q^{\delta}_{x,y}$ the law, on $C_{+}([0,1])$, of the $\delta$-dimensional squared Bessel bridge between $x$ and $y$ on the interval $[0,1]$ (see Chapter XI.3 of \cite{revuz2013continuous} for the precise definition). Moreover, for any $\delta, a, b \geq 0$, we shall denote by $P^{\delta}_{a,b}$ the law, on $C_{+}([0,1])$, of the $\delta$-dimensional Bessel bridge between $a$ and $b$, and by $E^{\delta}_{a,b}$ the associated expectation operator. Recall that, by definition, $P^{\delta}_{a,b}$ is the image of $Q^{\delta}_{a^{2},b^{2}}$ under the map 

\begin{align}
\label{sqrt_map}
\begin{cases}
C([0,1], \mathbb{R}_{+}) \to C([0,1], \mathbb{R}_{+}) \\
X \mapsto \sqrt{X}.
\end{cases}
\end{align}
We will use the shorthand notations $Q^{\delta}$ and $P^{\delta}$ for $Q^{\delta}_{0,0}$ and $P^{\delta}_{0,0}$, respectively. The family of probability measures $\left(Q^{\delta}_{x,0}\right)_{\delta, x \geq 0}$ satisfies a remarkable additivity property. To state it, we introduce the following:

\begin{df}
For any  two laws $\mu, \nu$ on $C_{+}([0,1])$, let $\mu \ast \nu$ denote the convolution of $\mu$ and $\nu$, i.e. the image of $\mu \otimes \nu$ under the map of addition of paths:
\[ C_{+}([0,1]) \times C_{+}([0,1]) \to C_{+}([0,1]), \quad (x,y) \mapsto x+y. \]
\end{df}

\noindent The following statement is a particular case of Theorem 5.8 in  \cite{pitman1982decomposition}:

\begin{prop}
\label{levy}
For all $x,x', \delta, \delta'$, we have the following equality of probability laws on $C_{+}([0,1])$:
\[ Q^{\delta}_{x,0} \ast Q^{\delta'}_{x',0} = Q^{\delta + \delta'}_{x + x',0}. \]
\end{prop}

\noindent Note that this is an equivalent for the bridges of the well-known additivity property of squared Bessel processes first observed by Shiga and Watanabe in \cite{shiga1973bessel}. Note also that this relation in particular says that the families of probability measures $\left( Q^{0}_{x,0} \right)_{x \geq 0}$ and $\left( Q^{\delta} \right)_{\delta \geq 0}$ are convolution semi-groups on $C_{+}([0,1])$. In \cite{pitman1982decomposition}, the authors constructed the corresponding L\'{e}vy measures $M_{0}$ and $N_{0}$ on $C_{+}([0,1])$ (they actually provided an explicit construction in the case of the unconstrained squared Bessel bridges, but stressed that the case of the bridges can be dealt similarly, see section (5.4) in that article). The measures $M_{0}$ and $N_{0}$ are characterized by the fact that $M_{0}(\{ 0 \}) = N_{0}(\{ 0 \}) =0$ and, for all $\delta, x \geq 0$ and all $\theta : [0,1] \to \mathbb{R}_{+}$ bounded and Borel, we have
\begin{align}
\label{levy_khintchine_bridges}
Q^{\delta}_{x,0} \left[ \exp \left( - \langle \theta,  X \rangle \right) \right] = &\exp \left( - x \int \left( 1 - \exp \left( - \langle \theta, X \rangle \right) \right) dM_{0}(X) \right) \\ 
\nonumber &\exp \left( - \delta \int \left( 1 - \exp \left( - \langle \theta, X \rangle \right) \right) dN_{0}(X) \right). 
\end{align}

\subsection{Laws of pinned squared Bessel bridges as a convolution semigroup on $C_{+}([0,1])$}

For all $\delta \geq 0$, $x \geq 0$ and $r \in (0,1)$, we denote by $ Q^{\delta} \left[ \ \cdot \ | X_{r} = x \right]$ the law of a squared Bessel bridge between $0$ and $0$ conditioned on the event $\{X_{r}=x\}$ (to which we shall also refer as the law of a \textit{pinned squared Bessel bridge}). Such a conditioning is degenerate, but we can give a canonical meaning to it using the Markov structure of squared Bessel bridges (see chapter XI.3 in \cite{revuz2013continuous}). Note that $Q^{\delta} \left[ \ \cdot \ | X_{r} = x \right]$ is the image of $Q^{\delta}_{x,0} \otimes Q^{\delta}_{x,0}$ under the reversal, scaling, and concatenation map $S_{r}:  C_{+}([0,1]) \times C_{+}([0,1]) \to C_{+}([0,1])$ defined, for all $X, Y \in C_{+}([0,1])$, by
\begin{align}
\label{map_sr}
S_{r} (X,Y) : 
\tau \mapsto
\begin{cases} & r X \left( \frac{r-\tau}{r} \right), \quad \text{if} \ 0 \leq \tau \leq r \\
& (1-r) Y \left(\frac{\tau-r}{1-r} \right), \quad \text{if} \ r < \tau \leq 1.
\end{cases} 
\end{align} 
With this representation, we see that Proposition \ref{levy} implies the following:

\begin{prop}
\label{additivity_pinned}
For all $r \in (0,1)$ and all $x,x', \delta, \delta'$, we have the following equality of probability laws on $C_{+}([0,1])$:
\[ Q^{\delta} \left[ \ \cdot \ | X_{r} = x \right] \ast Q^{\delta'} \left[ \ \cdot \ | X_{r} = x' \right] = Q^{\delta + \delta'} \left[ \ \cdot \ | X_{r} = x + x' \right]. \]
\end{prop}

A very important consequence for us will be the fact that $\left( Q^{0} \left[ \ \cdot \ | X_{r} = x \right] \right)_{x \geq 0}$ forms a convolution semigroup of probability laws on $C_{+}([0,1])$. Exploiting the constructions of Pitman-Yor, we can furthermore exhibit the associated L\'{e}vy measure:

\begin{prop} 
\label{levy_meas_pinned_sqred_bridges}
Let $r \in (0,1)$. There exists a measure $M^{r}$ on $C_{+}([0,1])$ such that $M^{r}(\{ 0 \}) = 0$ and, for all $x \geq 0$ and all $\theta : [0,1] \to \mathbb{R}_{+}$ bounded and Borel, we have
\begin{equation}
\label{Levy_meas}
 Q^{0} \left[ \exp \left( - \langle \theta,  X \rangle \right) | X_{r} =x \right] = \exp \left( - x \int \left( 1 - \exp \left( - \langle \theta, X \rangle \right) \right) dM^{r}(X) \right).
\end{equation}
\end{prop} 

\begin{proof}

Let $x \geq 0$ and $\theta : [0,1] \to \mathbb{R}_{+}$ bounded and Borel.
Since $Q^{\delta} \left[ \ \cdot \ | X_{r} = x \right]$ is the image of $Q^{\delta}_{x,0} \otimes Q^{\delta}_{x,0}$ under the map $S_{r}$ defined by \eqref{map_sr}, we have
\begin{align*}
&Q^{0} [\exp (- \langle \theta, X \rangle | X_{r} = x] = \\
&Q^{0}_{\frac{x}{r}, 0} \left[\exp \left(- \int_{0}^{1} \underline{\theta} (1-v) \ X_{v} \ dv\right) \right] \ Q^{0}_{\frac{x}{1-r}, 0} \left[\exp \left(- \int_{0}^{1} \overline{\theta} (v) \ X_{v} \ dv\right) \right],
\end{align*}
where 
\[ \underline{\theta} (v) := r^{2} \ \theta (r v), \quad 0 \leq v \leq 1, \]
and
\[ \overline{\theta} (v) := (1-r)^{2} \ \theta \left( r +v(1-r) \right), \quad 0 \leq v \leq 1. \]
Therefore, by \eqref{levy_khintchine_bridges}, we obtain
\begin{align*}
Q^{0} [\exp (- \langle \theta, X \rangle | X_{r} = x] =  \exp \Big[ & - \frac{x}{r} \int \left( 1 - \exp \left( - \int_{0}^{1} \underline{\theta} (1-v) X_{v} dv \right) \right) dM_{0}(X) \\
& - \frac{x}{1-r} \int \left( 1 - \exp \left( - \int_{0}^{1} \overline{\theta} (v) X_{v} du \right) \right) dM_{0}(X) \Big] .
\end{align*}
Upon performing the changes of variable $u := r(1-v)$ in the first integral, and $u:= r + v(1-r)$ in the second one, this yields
\begin{align*}
Q^{0} [\exp (- \langle \theta, X \rangle | X_{r} = x] &= \exp \Bigr[ - \frac{x}{r} \int \left( 1 - \exp \left( - \int_{0}^{r} \theta(u) \ r X_{\frac{r-u}{r}} du \right) \right) dM_{0}(X)  \\
&- \frac{x}{1-r} \int \left( 1 - \exp \left( - \int_{r}^{1} \theta (u) \ (1-r) X_{\frac{u-r}{1-r}} du \right) \right) dM_{0}(X) \Bigr] .
\end{align*}
Therefore, denoting by $M^{r}_{1}$ the image of $M_{0}$ under the map
\begin{align*}
\begin{cases}
C_{+}([0,1]) &\to C_{+}([0,1]) \\
X &\mapsto \left(r X_{\frac{r-u}{r}} \mathbf{1}_{[0,r]}(u) \right)_{0 \leq u \leq 1},
\end{cases}
\end{align*}
and by $M^{r}_{2}$ the image of $M_{0}$ under the map
\begin{align*}
\begin{cases}
C_{+}([0,1]) &\to C_{+}([0,1]) \\
X &\mapsto \left((1-r) X_{\frac{u-r}{1-r}} \mathbf{1}_{[r,1]}(u)\right) ,
\end{cases}
\end{align*}
and setting $M^{r} := \frac{1}{r}M^{r}_{1} + \frac{1}{1-r}M^{r}_{2}$, we deduce that $M^{r}(\{ 0 \}) =  0$, and that \eqref{Levy_meas} holds.
\end{proof}

The above Propositions will be very important for us in proving Taylor estimates for the laws of pinned Bessel bridges $P^{\delta} \left[ \ \cdot \ | X_{r} = b \right]$, for $r \in (0,1)$ and $\delta, b \geq 0$. We recall that $P^{\delta} \left[ \ \cdot \ | X_{r} = b \right]$ is the image of $Q^{\delta} \left[ \ \cdot \ | X_{r} = b^{2} \right]$ under the square root map \eqref{sqrt_map}. 

\section{Density of $\mathscr{S}$ in a large space of functionals on $L^{2}(0,1)$}
\label{section_density}

In this section we prove that a large class of functionals $\Phi : L^{2}(0,1) \to \mathbb{R}$ can be approximated by elements of $\mathscr{S}$. We do not need convergence in a very strong sense: point-wise convergence with some uniform dominations on the functionals and their differentials will suffice for our purpose. More precisely, we introduce the following definition:

\begin{df}
\label{pdi_conv}
Let $p \in \{1,2\}$, and let $\Phi_{n} \ (n \geq 1)$ and $\Phi$ be functionals on $L^{2}(0,1)$ which are differentiable at each element of $C_{+}([0,1])$, along any direction in $C^{2}_{c}(0,1)$. 
We say that the convergence assumption $(A_p)$ is satisfied if the following conditions hold:
\begin{enumerate}[label=(\roman*)]
\item \label{fst_condition} for all $X \in C_{+}([0,1])$, we have the convergence 
\[ \Phi_{n}(X) \underset{n \to \infty}{\longrightarrow} \Phi(X), \]
together with the domination 
\[ \forall n \geq 1, \quad |\Phi_{n}(X)| \leq \|\Phi\|_{\infty},\]   
\item \label{snd_condition} for all $h \in C^{2}_{c}(0,1)$, and all $X \in C_{+}([0,1])$, we have the convergence 
\[ \partial_{h} \Phi_{n} (X) \underset{n \to \infty}{\longrightarrow} \partial_{h} \Phi(X), \]
together with the domination
\[ \forall n \geq 1, \quad |\partial_{h} \Phi_{n} (X)| \leq C \|h\|_{\infty} (1+\|X\|), \]
where $C>0$ is some contant,  
\item \label{thd_condition} there exists $K>0$ such that, for all $X, Y \in  C_{+}([0,1])$ and $n \geq 1$, we have
\[ |\Phi_n(X) -  \Phi_n(Y)| \leq K \|X^{2} - Y^{2}\|_{1}^{1/p}.\]  
\end{enumerate}
\end{df}

\begin{prop}
\label{density_l1}
Let $\Phi \in \mathcal{S}C^{1}_{b}\left( L^{1}(0,1) \right)$. Then there exists a family $(\Phi^{d}_{n,k})_{d,n,k \geq 1}$ of elements of $\mathscr{S}$ such that
\begin{equation}
\label{conv_approx_sequences}
\underset{d \to \infty}{\lim} \ \underset{n \to \infty}{\lim} \ \underset{k \to \infty}{\lim} \ \Phi^{d}_{n,k} = \Phi, 
\end{equation}
where all the convergences hold in the sense of $(A_1)$.
\end{prop}

\begin{rk}
We stress that, in the above statement, the domination properties associated with assumption $(A_1)$ (see Definition \ref{pdi_conv}) are uniform only on one index, the other indices being fixed. For instance, for all $d,n \geq  1$, there exists $C(d,n) > 0$ such that
\[\forall k \geq 1, \quad |\partial_{h} \Phi^{d}_{n,k}(X) | \leq C(d,n) \|h\|_{\infty}  \|X\|, \]
but we do not claim that the constants $C(d,n)$ are bounded uniformly in $d,n \geq 1$. However, such bounds will be sufficient for our purposes; indeed, the only reason we need them is in order to show that each term in the IbPF converges when we take the successive limits $k \to \infty$, $n \to \infty$ and $d \to \infty$.  The domination properties stated above will precisely allow us to do that by applying the dominated convergence theorem three times, successively.
\end{rk}

\begin{proof}
We will proceed in three steps, by constructing sequences $(\Phi^{d})_{d,\geq 1}$, $(\Phi^{d}_{n})_{d,n\geq 1}$ and $(\Phi^{d}_{n,k})_{d,n,k \geq 1}$ of functionals on $L^{2}(0,1)$ such that $\Phi^{d}_{n,k} \in \mathscr{S}$ for all $d,n,k \geq 1$, 
with the following convergences in the sense of $(A_{1})$:
\[ \Phi^{d}_{n,k} \underset{k \to \infty}{\longrightarrow} \Phi^{d}_{n} \underset{n \to \infty}{\longrightarrow}  \Phi^{d} \underset{d \to \infty}{\longrightarrow}  \Phi. \]
We start by constructing $(\Phi^{d})_{d \geq 1}$. Let $\Psi \in C^{1}_{b}\left( L^{1}(0,1) \right)$ such that $\Phi(X) = \Psi(X^{2})$ for all $X \in L^{2}(0,1)$. Then, for any $d \geq 1$, we define $(\zeta^{d}_{i})_{1 \leq i \leq d}$ to be the orthonormal family in $L^{2}(0,1)$ given by
\begin{equation}
\label{zeta_d_i}
 \zeta^{d}_{i} := \sqrt{d} \ \mathbf{1}_{[\frac{i-1}{d}, \frac{i}{d})}, \quad i = 1, \ldots, d, 
\end{equation}
and we define $\Phi^{d}$ by
\[ \Phi^{d} (X) := \Psi \left( \sum_{i=1}^{d} \langle \zeta^{d}_{i}, X^{2} \rangle \ \zeta^{d}_{i}  \right), \quad X \in L^{2}(0,1). \]
We check that $\Phi^{d}$ converges to $\Phi$ in the sense of $(A_1)$ as $d \to \infty$. We first remark that, for all $X \in C([0,1])$, we have
\[ \sum_{i=1}^{d} \langle \zeta^{d}_{i}, X^{2} \rangle \ \zeta^{d}_{i} \underset{d \to \infty}{\longrightarrow} X^{2} \]
uniformly on $(0,1)$, hence in particular in $L^{1}(0,1)$. Since $\Psi : L^{1}(0,1) \to \mathbb{R}$ is continuous, this implies that 
\[ \Phi^{d}(X)  \underset{d \to \infty}{\longrightarrow} \Phi(X). \]
Moreover, we have the domination
\[ \forall X \in L^{2}(0,1), \quad |\Phi^{d}(X)| \leq \| \Phi \|_{\infty}, \]
as requested by condition \ref{fst_condition} in Definition \ref{pdi_conv}. Furthermore, for all $h \in C^{2}_{c}(0,1)$ and $X \in C([0,1])$, we have
\begin{equation*}
 \partial_{h} \Phi^{d}(X) = 2  D \Psi \left( \sum_{i=1}^{d} \langle \zeta^{d}_{i}, X^{2} \rangle \ \zeta^{d}_{i}  \right) \left( \sum_{i=1}^{d} \langle \zeta^{d}_{i} X, h \rangle \zeta^{d}_{i} \right). 
\end{equation*}
Now, since $\Psi$ is $C^{1}$ on $L^{1}(0,1)$, we have
\[ D \Psi \left( \sum_{i=1}^{d} \langle \zeta^{d}_{i}, X^{2} \rangle \ \zeta^{d}_{i}\right) \underset{d \to \infty}{\longrightarrow} D \Psi(X^{2}) \quad \text{in} \ L^{1}(0,1)', \]
while, at the same time, we also have 
\[ \sum_{i=1}^{d} \langle \zeta^{d}_{i} X, h \rangle \ \zeta^{d}_{i} \underset{d \to \infty}{\longrightarrow} h \ X, \]
 uniformly in $(0,1)$, hence in $L^{1}(0,1)$. Therefore
\[   D \Psi \left( \sum_{i=1}^{d} \langle \zeta^{d}_{i}, X^{2} \rangle \ \zeta^{d}_{i}  \right) \left( \sum_{i=1}^{d} \langle \zeta^{d}_{i} X, h \rangle \zeta^{d}_{i} \right) \underset{d \to \infty}{\longrightarrow}  2 D \Psi(X^{2})(hX),  \]
i.e. 
\[ \partial_{h} \Phi^{d}(X)  \underset{d \to \infty}{\longrightarrow}  \partial_{h} \Phi(X).  \]
Moreover, for all $d \geq 0$, we have
\begin{equation}
\label{bound_der_phi_d}
 | \partial_{h} \Phi^{d}(X) | \leq 2 ||| D \Psi |||_{\infty} \| h \|_{\infty} \| X \|, 
\end{equation} 
which provides the requested domination property for $\left(\partial_{h} \Phi^{d} \right)_{d \geq 1}$. Thus condition \ref{snd_condition} is fulfilled as well. Finally, for all $X, Y \in C_{+}([0,1])$, we have
\begin{align*}
\left| \Phi^{d}(X) -  \Phi^{d}(Y) \right| &= \left| \Psi(X^{2}) - \Psi(X^{2}) \right| \\
&\leq ||| D \Psi |||_{\infty} \|X^{2} - Y^{2} \|_{1},
\end{align*}
as requested by condition \ref{thd_condition}. Therefore, $\Phi^{d} \underset{d \to \infty}{\longrightarrow} \Phi$ in the sense of $(A_1)$.  

We now fix $d \geq 1$ and, for all integer $n \geq 1$, we construct $\Phi^{d}_{n}$. The latter will be a truncated version of $\Phi^{d}$ obtained as follows. Let $\chi: \mathbb{R} \to \mathbb{R}$ be a smooth function with values in $[0,1]$, such that $\chi =1$ on $(- \infty, -1]$ and $\chi = 0$ on $[0, +\infty)$. Set $\chi_{n}( \cdot) := \chi( \cdot -n)$ and let
\[ \Phi^{d}_{n}(X) = \Phi^{d}(X) \prod_{i=1}^{d} \chi_{n} \left( \langle \zeta_{i}, X^{2} \rangle \right), \quad X \in L^{2}(0,1). \]
We check that $\Phi^{d}_{n}$ converges to $\Phi^{d}$ in the sense of $(A_1)$ as $n \to \infty$. Since $\chi_{n} \underset{n \to \infty}{\longrightarrow} 1$ pointwise, we have 
\[ \Phi^{d}_{n}(X) \underset{n \to \infty}{\longrightarrow} \Phi^{d}(X) \]
for all $X \in L^{2}(0,1)$. Moreover, we have $|\Phi^{d}_{n}(X)| \leq \| \Phi^{d} \|_{\infty}$ for all $n \geq 1$ and $X \in L^{2}(0,1)$.  Hence, the convergence and domination assumptions in condition \ref{fst_condition} do indeed hold. Turning to condition \ref{snd_condition}, we remark that, for all $n\in \mathbb{N}$, $h \in C^{2}_{c}(0,1)$ and $X \in L^{2}(0,1)$, we have
\begin{align}
\label{expression_der_phi_n}
\partial_{h} \Phi^{d}_{n}(X) &= \partial_{h} \Phi^{d}(X) \prod_{i=1}^{d} \chi_{n} \left( \langle \zeta_{i}, X^{2} \rangle \right) \\
\nonumber &+ \Phi^{d}(X) \sum_{i=1}^{d} \chi_{n}' \left( \langle \zeta_{i}, X^{2} \rangle \right) \prod_{j \neq i} \chi_{n} \left( \langle \zeta_{j}, X^{2} \rangle \right)  \langle 2\zeta_{i} X, h \rangle
\end{align}
Since $\chi_{n} \underset{n \to \infty}{\longrightarrow} 1$ and $\chi_{n}' \underset{n \to \infty}{\longrightarrow} 0$ pointwise, it holds that $\partial_{h} \Phi^{d}_{n}(X) \underset{n \to \infty}{\longrightarrow} \partial^{d}_{h} \Phi(X)$. Moreover, by equality \eqref{expression_der_phi_n} and the Leibniz formula, and recalling \eqref{bound_der_phi_d}, we have
\[ | \partial_{h} \Phi^{d}_{n} (X) | \leq 2 \| \Psi \|_{C^{1}} \left( 1 + d \|\chi'\|_{\infty}\right) \| h \|_{\infty}  \| X \|, \]
which provides the requested domination property. Finally, for all $n \geq 1$ and $X, Y \in C_{+}([0,1])$, we have
\begin{align*}
| \Phi^{d}_{n}(X) - \Phi^{d}_{n}(Y)| &\leq \| \Psi \|_{C^{1}}  \left(1 + \sum_{i=1}^{d} \|\chi'\|_{\infty} \|\zeta^{d}_{i}\|_{\infty} \right) \|X^{2} - Y^{2}\|_{1} \\
&\leq  \| \Psi \|_{C^{1}}  (1 + d^{3/2} \|\chi'\|_{\infty} ) \, \|X^{2} - Y^{2}\|_{1},
\end{align*}
so condition \ref{thd_condition} is fulfilled as well. Hence $\Phi^{d}_{n}$ converges in the sense of $(A_1)$ to $\Phi^{d}$ as $n \to \infty$.

Finally, we fix $d, n \geq 1$, and construct the sequence $(\Phi^{d}_{n,k})_{k \geq 1}$. Note that $\Phi^{d}_{n}$ is of the form
\[ \Phi^{d}_{n} (X) = h \left( \langle \zeta_{1}, X^{2} \rangle, \ldots, \langle \zeta_{d}, X^{2} \rangle \right), \quad X \in L^{2}(0,1), \]
where $h: \mathbb{R}_{+}^{d} \to \mathbb{R}$ is the function given by 
\begin{equation}
\label{def_gn}
 h(x) := \Psi \left( \sum_{i=1}^{d} x_{i} \ \zeta^{d}_{i} \right) \prod_{i=1}^{d} \chi (x_{i} -n), \quad x \in \mathbb{R}_{+}^{d}. 
\end{equation}
Remark that $h$ is $C^{1}$ with bounded support in $[0,n]^{d}$. We now make use of an approximation result on $\mathbb{R}_+^d$. Denoting by $\cdot$ the standard inner product on $\mathbb{R}^{d}$, let $\mathscr{E}$ be the linear span of the functions
\[e^{- \lambda \cdot} : 
\begin{cases} 
\mathbb{R}_+^d \to \mathbb{R} \\
x \mapsto e^{- \lambda \cdot x} \end{cases}, \]
for $\lambda \in \mathbb{R}_+^d$. We state the following approximation result, the proof of which is postponed to the Appendix (Section \ref{sect_appendix}):

\begin{lm}
\label{pseudo_Weierstrass}
Given $n \geq 1$ a fixed integer, let  $h : \mathbb{R}_{+}^{d} \to \mathbb{R}$ be a $C^{1}$ function supported  in  $[0,n]^{d}$. Then there exists a sequence of functions $h_k \in \mathscr{E}$, $k \geq 1$, such that:
\begin{itemize}
\item for all $x \in \mathbb{R}_{+}^{d}, \quad h_{k}(x) \underset{k \to \infty}{\longrightarrow} h(x)$ and $ \nabla h_{k}(x) \underset{k \to \infty}{\longrightarrow} \nabla h(x)$,
\item for all $k \geq 1$ and all $x \in \mathbb{R}_{+}^{d}$, we have 
\[ |h_{k}(x)| \leq |h(x)|,  \]
and
\[ \forall i = 1 \ldots d, \qquad  |\partial_{i} h_{k}(x)| \leq C(n) \ | \partial_{i} h(x)|, \]
where $C(n)$ is a positive constant depending only on $n$.
\end{itemize}
\end{lm}

For all $n \geq 1$ fixed, let now $(h_k)_{k \geq 1}$ be a sequence of elements of $\mathscr{E}$ approximating the function $h$ defined in \eqref{def_gn} as in Lemma \ref{pseudo_Weierstrass},  and set
\[ \Phi^{d}_{n,k}(X) := h \left( \langle \zeta_{1}, X^{2} \rangle, \ldots, \langle \zeta_{d}, X^{2} \rangle \right), \quad X \in L^{2}(0,1). \]
Then for all $k \geq 1$, the functional $\Phi^{d}_{n,k}$ lies in $\mathscr{S}$. Moreover, using the properties of $(h_k)_{k \geq 1}$ one easily deduces that the sequence $\left(\Phi^{d}_{n,k}\right)_{k \geq 1}$ 
converges in the sense of $(A_1)$ to $\Phi^{d}_{n}$. This yields the claim. 
\end{proof}

In the proof of the IbPF for $\delta \in (0,1)$, we shall need a slight refinement of the above proposition stating that, if $\Phi \in \mathcal{S}C^{1,1}_{b}\left( L^{1}(0,1) \right)$, the approximating sequences converge in a stronger sense. More precisely, we introduce the following notion of convergence:

\begin{df}
\label{pdi_lip_conv}
Let $\Phi_{n} \ (n \geq 1)$ and $\Phi$ be functionals on $L^{2}(0,1)$ which are differentiable at each element of $C_{+}([0,1])$, along any direction in $C^{2}_{c}(0,1)$. 
We say that the sequence $\left( \Phi_{n} \right)_{n \geq 1}$ and $\Phi$ satisfy the convergence assumption $(A_1^+)$ if both of the following conditions hold:
\begin{itemize}
\item $\left( \Phi_{n} \right)_{n \geq 1}$ converges to $\Phi$ in the sense of $(A_1)$,
\item for all $ n \geq 1$ and $ X, Z, Z \in L^{1}_{+}(0,1)$, we have
\begin{equation}
\label{condition_pdi_lip}
\begin{split}
&\left|\Phi_{n}(\sqrt{X+Z+Z'}) - \Phi_{n}(\sqrt{X+Z}) - \Phi_{n}(\sqrt{X+Z'}) + \Phi_{n}(\sqrt{X})\right| \\
&\leq L \|Z\|_{1} \|Z'\|_{1},
\end{split}
\end{equation}
where $L>0$ is some constant.
\end{itemize}
\end{df} 

\begin{prop}
\label{density_l1_lip}
Let $\Phi \in \mathcal{S} C^{1,1}_{b}(L^{1}(0,1))$. Then the approximating sequences of functionals given by Proposition \ref{density_l1}  are such that the convergences \eqref{conv_approx_sequences} actually hold in the sense of $(A_1^+)$.
\end{prop}


\begin{proof}
Since we already know that the convergence \eqref{conv_approx_sequences} holds in the sense of $(A_1)$, there only remains to prove that these approximating sequences further satisfy condition \eqref{condition_pdi_lip}. Let $\Psi$ as in \eqref{phi_psi_lip}. Since $\Psi \in C^{1,1}_{b}(B)$, there exists $L>0$ satisfying \eqref{lipschitz_differential}. Moreover, since the map 
\begin{align*}
\begin{cases}
L^{1}(0,1) &\to \quad L^{1}(0,1) \\
Z &\mapsto \quad  \sum_{i=1}^{d} \langle  \zeta^{d}_{i}, Z \rangle \zeta^{d}_{i}
\end{cases}
\end{align*}
is Lipschitz continous (with Lipschitz constant $1$), we deduce that the functional $\Psi^{d}: Z \mapsto \Psi \left( \sum_{i=1}^{d} \langle  \zeta^{d}_{i}, Z \rangle \zeta^{d}_{i} \right) $ also satisfies \eqref{lipschitz_differential}. As a consequence, by Remark \ref{bound_double_order_incr}, for all $X, Z, Z' \in L^{1}_{+}(0,1)$ and $d \geq 1$, we have
\begin{align*}
& \left| \Phi^{d}(\sqrt{X+Z+Z'}) - \Phi^{d}(\sqrt{X+Z}) - \Phi^{d}(\sqrt{X+Z'}) + \Phi^{d}(\sqrt{X}) \right| \\
&= \left| \Psi^{d}(X+Z+Z') - \Psi^{d}(X+Z) - \Psi^{d}(X+Z') + \Psi^{d}(X) \right| \\
&\leq L \left\| Z \right\|_{1} \left\| Z'\right\|_{1}.
\end{align*}
Hence, the sequence $(\Phi^{d})_{d \geq 1}$ satisfies the condition \eqref{condition_pdi_lip}, so it converges in the sense of $(A_1^+)$ to $\Phi$. Moreover, for all $d \geq 1$, $\Psi^{d} \in C^{1,1}_{b}(L^1(0,1))$ and $\chi'$ is globally Lipschitz (it is smooth and compactly supported). Hence, for all $n \geq 1$, the functional $\Psi^{d}_{n}$ given by
\[ \Psi^{d}_{n}(Z) :=  \Psi^{d}(Z) \prod_{i=1}^{d} \chi_{n} \left( \langle \zeta_{i}, Z \rangle \right), \quad Z \in L^{1}(0,1), \]
satisfies  \eqref{lipschitz_differential}, with some Lipschitz constant $L'$ depending only on $\Psi$, $\chi$ and $d$. Therefore for all $n \geq 1$ and $X,Z,Z' \in C_{+}([0,1])$, we have
\begin{align*}
&\left| \Phi^{d}_{n}(\sqrt{X+Z+Z'}) - \Phi^{d}_{n}(\sqrt{X+Z}) - \Phi^{d}_{n}(\sqrt{X+Z'}) + \Phi^{d}_{n}(\sqrt{X}) \right| \\
&= \left| \Psi^{d}_{n}(X+Z+Z') - \Psi^{d}_{n}(X+Z) - \Psi^{d}_{n}(X+Z') + \Psi^{d}_{n}(X) \right| \\
&\leq L' \left\| Z \right\|_{1} \left\| Z'\right\|_{1},
 \end{align*}
so $(\Phi^{d}_{n})_{n \geq 1}$ satisfies the condition \eqref{condition_pdi_lip}, and hence converges in the sense of $(A_1^+)$ to $\Phi$. Finally, there remains to prove that, for all fixed $d,n \geq 1$, the sequence $\left(\Phi^{d}_{n,k}\right)_{k \geq 1}$ satisfies the condition \eqref{condition_pdi_lip}. To do so, note that, for all $n \geq 1$, the function $h$ defined by \eqref{def_gn} satisfies
\[  \forall i = 1, \ldots, d, \ \forall x, y \in \mathbb{R}_{+}^{d}, \quad |\partial_{i}h (x) - \partial_{i}h(y) | \leq L' \sum_{j=1}^{d} |x_{j} - y_{j}|, \]
where $L'>0$ is as above. By Lemma \ref{pseudo_Weierstrass_bis} below, this implies that, for all $n \geq 1$, the sequence $(h_k)_{k \geq 1}$ of functions approximating $h$ as in Lemmas \ref{pseudo_Weierstrass} satisfies the bound
\begin{align*} 
\forall x, z, z' \in \mathbb{R}_{+}^{d}, \quad
&|h_k(x+z+z') - h_k(x+z) - h_k(x+z') + h_k(x)| \\
&\leq C'(n)  \left( L' + \| \partial_{i} h \|_{\infty}  \right) \, \sum_{j=1}^{d} |z_{j}|  \sum_{j=1}^{d} |z'_{j}| . 
\end{align*}
From that inequality we deduce that, for all $X, Z, Z' \in L^{1}_{+}(0,1)$, we have
\begin{align*}
&\left| \Phi^{d}_{n,k}(\sqrt{X+Z+Z'}) - \Phi^{d}_{n,k}(\sqrt{X+Z}) - \Phi^{d}_{n,k}(\sqrt{X+Z'}) + \Phi^{d}_{n,k}(\sqrt{X}) \right| \\
&\leq C'(n) \left( L' + \| \partial_{i} h \|_{\infty}  \right) \|Z\|_{1} \|Z'\|_{1}, 
 \end{align*}
which proves that the sequence $\left(\Phi^{d}_{n,k}\right)_{k \geq 1}$ satisfies the condition \eqref{condition_pdi_lip}, and hence converges to $\Phi^{d}_{n}$ in the sense of $(A_1^+)$. This yields the claim
\end{proof}

In the above proof we used the following Lemma, the proof of which is postponed to the Appendix (Section \ref{sect_appendix}):
\begin{lm}
\label{pseudo_Weierstrass_bis}
Given $n \geq 1$ a fixed integer, let  $h : \mathbb{R}_+^d \to \mathbb{R}$ be a $C^{1}$ function supported  in  $[0,n]^{d}$, and satisfying furthermore:
\begin{equation}
\label{lip_deriv_g}
    \forall i = 1, \ldots, d, \ \forall x, y \in \mathbb{R}_{+}^{d}, \quad |\partial_{i}h (x) - \partial_{i}h(y) | \leq L' \sum_{j=1}^{d} |x_{j} - y_{j}|
\end{equation}  
for some constant $L'>0$. Then the sequence of functions $\left( h_{k} \right)_{k \geq 0}$ given by Lemma \ref{pseudo_Weierstrass} further satisfies the following: for all $k \geq 1$ and $i = 1, \ldots, d$, we have
\[ \forall x, y  \in \mathbb{R}_{+}^{d}, \quad | \partial_{i} h _{k} (x) - \partial_{i} h _{k} (y) | \leq C'(n)  \left( L' + \| \partial_{i} h \|_{\infty}  \right) \sum_{j=1}^{d} |x_{j} - y_{j}| \]
where $C'(n)>0$ is a constant depending only on $n$.
\end{lm}


Propositions \ref{density_l1} and \ref{density_l1_lip} enable to approximate, by elements of $\mathscr{S}$, any functional $\Phi$ of the form  
\[ \Phi(X) = \Psi(X^{2}), \quad X \in L^{2}(0,1), \]
where $\Psi \in C^{1}_{b}\left(L^{1}(0,1)\right)$. Such an assumption may appear rather restrictive, since it in particular forces $D\Psi(0)$ to vanish. However, it turns out that for general functionals $\Phi \in C^{1}_{b}(L^{2}(0,1))$, we can also obtain such an approximation result, but in a weaker sense. 

\begin{prop}
\label{density_l2}
Let $\Phi \in C^{1}_{b}\left(L^{2}(0,1)\right)$. Then there exists a family $(\Phi^{m,d}_{n,k})_{m,d,n,k \geq 1}$ of elements of $\mathscr{S}$ such that
\begin{equation}
\label{conv_approx_sequences_l2}
\underset{m \to \infty}{\lim} \ \underset{d \to \infty}{\lim} \ \underset{n \to \infty}{\lim} \ \underset{k \to \infty}{\lim} \ \Phi^{m,d}_{n,k} = \Phi 
\end{equation}
where the first three limits are in the sense of $(A_1)$, while the last limit (in $m$) is in the sense of $(A_2)$.
\end{prop}

\begin{proof}

As in the proof of Proposition \ref{density_l1}, we will proceed in several steps, by constructing sequences $(\Phi^{m})_{m \geq 1}$, $(\Phi^{m,d})_{d,\geq 1}$, $(\Phi^{m,d}_{n})_{m, d,n\geq 1}$ and $(\Phi^{m,d}_{n,k})_{m,d,n,k \geq 1}$ of functionals on $L^{2}(0,1)$ such that $\Phi^{m,d}_{n,k} \in \mathscr{S}$ for all $m,d,n,k \geq 1$, with 
\[ \Phi^{m,d}_{n,k} \underset{k \to \infty}{\longrightarrow} \Phi^{m,d}_{n} \underset{n \to \infty}{\longrightarrow}  \Phi^{m,d} \underset{d \to \infty}{\longrightarrow}  \Phi^{m} \underset{m \to \infty}{\longrightarrow} \Phi,\]
where the first three convergences hold in the sense of $(A_1)$, and the last one holds in the sense of $(A_2)$.

We start by constructing $(\Phi^{m})_{m\geq 1}$. For all $m \geq 1$, let $\Phi^{m}$ be the functional given by
\[ \Phi^{m} (X) := \Phi \left( \sqrt{ X^{2} + \frac{1}{m}} \right), \quad X \in L^{2}(0,1). \]
We show that the sequence $\left(\Phi^{m}\right)_{m \geq 1}$ converges in the sense of $(A_2)$ to $\Phi$. 
It is easy to check that conditions \ref{fst_condition} and \ref{snd_condition} of Definition \ref{pdi_conv} are satisfied, so we focus on 
the proof of condition $(iii)$. To this end, we use the fact that, for all $u,v \in L^2_+(0,1)$,
\begin{equation}
\label{eqn.l2.l1}
\|u-v\|_2 \leq \|u^2 - v^2\|_1^{1/2}, 
\end{equation}
which follows from the inequality $(a-b)^2 \leq |a^2-b^2|$ valid for all $a,b \geq 0$. Hence, for all $X,Y \in C_{+}([0,1])$,
$$\left\|\sqrt{X^2+\frac{1}{m}} - \sqrt{Y^2 + \frac 1m} \right\|_2 \leq \|X^2-Y^2\|_1^{1/2}, $$
whence 
\[  \left| \Phi^{m}(X) -  \Phi^{m}(Y) \right| \leq \| \nabla \Phi \|_{\infty} \left\| X^{2}-Y^{2} \right\|_{1}^{1/2},\]
which provides the domination condition \ref{thd_condition} with $p=2$.
Hence, the sequence $\left(\Phi^{m}\right)$ converges in the sense of $(A_2)$ to $\Phi$. 

For $m \geq 1$ fixed, the sequences $(\Phi^{m, d})_{m, d\geq 1}$, $(\Phi^{m, d}_{n})_{m, d,n\geq 1}$ and $(\Phi^{m, d}_{n,k})_{m,d,n,k \geq 1}$ can then be constructed from $\Phi^{m}$ in exactly the same way as $(\Phi^{d})_{d,\geq 1}$, $(\Phi^{d}_{n})_{d,n\geq 1}$ and $(\Phi^{d}_{n,k})_{d,n,k \geq 1}$ were constructed from $\Phi$ in the proof of Proposition \ref{density_l1}. The key remark is that, for all $X \in C_{+}(0,1)$, we have
\[ \Phi^{m}(X) = \Psi^{m}(X^{2}), \]
where 
\begin{align*}
\Psi^{m}: \begin{cases}
L^{1}(0,1) &\to L^{1}(0,1) \\
Z &\mapsto \Phi \left(\sqrt{|Z+\frac{1}{m}|} \right).
\end{cases}
\end{align*}
Although $\Psi^{m}$ is not $C^{1}$, it is Lipschitz continuous on $L^{1}(0,1)$, and, at each $X \in C_{+}([0,1])$, has directional derivatives in all directions $h \in C([0,1])$ satisfying the bound
\[ |\partial_{h} \Psi^{m}(Z)| \leq  \frac{m}{2} \| \Phi \|_{C^{1}} \| h \|. \]
Therefore, exactly as in the proof of Proposition \ref{density_l1}, we can show that the sequences $(\Phi^{m, d})_{m, d\geq 1}$, $(\Phi^{m, d}_{n})_{m, d,n\geq 1}$ and $(\Phi^{m, d}_{n,k})_{m,d,n,k \geq 1}$ will satisfy all the requested convergence and domination properties. We thus get the claim.

\end{proof}

 \section{Taylor estimates for the laws of pinned Bessel bridges}    
 
 \label{section_estimates}

In the previous section, we have shown that rather general functionals can be approximated by sequences of functionals in $\mathscr{S}$, for which we readily know that the IbPF derived in \cite{EladAltman2019} hold. Hence, to generalize the IbPF to the former functionals, we need to show that the terms appearing in our formulae converge when we take such limits. Thus in the case $\delta \in (1,3) $, as suggested by \eqref{onetothree}, we need to control, for all $r \in (0,1)$ and $b>0$, the quantity
\[ \mathcal{T}^{0}_{b} E^{\delta} [\Phi(X) | X_{r} = \cdot], \]
while in the case  $\delta \in (0,1)$, we need to control
\[\mathcal{T}^{2}_{b} E^{\delta} [\Phi(X) | X_{r} = \cdot], \]
for all sufficiently regular functional $\Phi$ on $L^{2}(0,1)$. Obtaining such estimates is the goal of the present section. 

\subsection{Taylor estimates at order $0$}

As recalled in the introduction, for all $\Phi \in \mathscr{S}$, $\delta \in (1,3)$ and $h \in C^{2}_{c}(0,1)$, the integral
\begin{equation}
\label{renormalized_integral}
 \int_{0}^{1} dr h(r) \int_{0}^{\infty} db \ p^{\delta}_{r}(b) \frac{1}{b^{3}} \mathcal{T}^{0}_{b} E^{\delta} [\Phi(X) | X_{r} = \cdot]
\end{equation}
is convergent. This is due to the fact that, for all $r \in (0,1)$, the function $ b \to E^{\delta} [\Phi(X) | X_{r} = \cdot]$ is smooth, with vanishing derivative at $0$. Hence, as $b \to 0$,
\[ \mathcal{T}^{0}_{b} E^{\delta} [\Phi(X) | X_{r} = \cdot] = O (b^{2}), \]
(see Remarks 4.2 and 4.3 in \cite{EladAltman2019}). By contrast, for an arbirary $\Phi \in C^{1}_{b}(L^{2}(0,1))$, it is not clear a priori whether such an estimate holds. Actually it is not even clear whether the integral \eqref{renormalized_integral} converges. However, it turns out that we can obtain a domination on the quantity $\mathcal{T}^{0}_{b} E^{\delta} [\Phi(X) | X_{r} = \cdot]$, even for an arbitrary $\Phi \in C^{1}_{b}(L^{2}(0,1))$. This bound is a little worse than in $b^{2}$, but it is still sufficient to make the double integral \eqref{renormalized_integral} converge.

\begin{prop}
\label{bound_lip_func}
There exists a universal constant $M>0$ such that the following holds: for all $\delta \geq 0$, all $L>0$ and all bounded and Borel measurable functional $\Phi: L^{2}(0,1) \to \mathbb{R}$ satisfying
\begin{equation}
\label{lipschitz_assumption}
\forall X, Y \in L^{2}_{+}(0,1), \quad | \Phi(X) - \Phi(Y) | \leq L  \left( \| X^{2} -Y^{2} \|_{1} \right)^{1/2}
\end{equation}
we have
\begin{equation}
\label{bound_cond_exp}
\forall r \in (0,1),  \forall b > 0, \quad |\mathcal{T}^{0}_{b} E^{\delta} [\Phi(X) | X_{r} = \cdot] | \quad \leq \quad M L \, b^{2} |\log(b)|.
\end{equation} 
In particular, for all such $\Phi $, and all $\delta > 1$, the function 
\[(r,b) \mapsto  \mathcal{T}^{0}_{b} E^{\delta} [\Phi(X) | X_{r} = \cdot] \]
 is integrable with respect to the measure $\frac{p^{\delta}_{r}(b)}{b^{3}} \, dr \, db$ on $(0,1) \times \mathbb{R}^{*}_{+}$.
\end{prop}

\begin{rk}
\label{bound_for_c1_funct}
Let $\Phi \in C^{1}_{b}(L^{2}(0,1))$. Then \eqref{lipschitz_assumption} holds with $L = \|\Phi\|_{C^{1}}$. Indeed, for all $X, Y \in L^{2}_{+}(0,1)$, we have
\begin{align*}
|\Phi(X) - \Phi(Y)| &\leq \|\Phi\|_{C^{1}} \| X - Y \| \\
&\leq \|\Phi\|_{C^{1}}   \left( \| X^{2} -Y^{2} \|_{1} \right)^{1/2}
\end{align*}
where the second inequality follows from \eqref{eqn.l2.l1}.
\end{rk}

\begin{proof}
Let $\Phi : L^{2}(0,1) \to \mathbb{R}$ satisfying \eqref{lipschitz_assumption}. We first assume the bound \eqref{bound_cond_exp} to be true and check that the second statement holds. Let $\delta > 1$. Recalling \eqref{density}, we have 
\begin{align*}
& \int_{0}^{1}\int_{0}^{\infty} \frac{p^{\delta}_{r}(b)}{b^{3}} |\mathcal{T}^{0}_{b} E^{\delta} [\Phi(X) | X_{r} = \cdot]| db \ dr \\
& \leq  \int_{0}^{1}\int_{0}^{\infty} \frac{ b^{\delta -4} }{2^{\frac{\delta}{2} -1} (r(1-r))^{\delta/2} \Gamma(\frac{\delta}{2})} \exp \left(- \frac{b^{2}}{2r(1-r)} \right) M L \, b^{2} |\log(b)| \  db \, dr \\
& = \frac{M L }{2^{\frac{\delta}{2} -1} \Gamma(\frac{\delta}{2})} \int_{0}^{1} \frac{1}{\left(r(1-r)\right)^{\delta/2}} \left( \int_{0}^{\infty} b^{\delta-2} |\log(b)| \exp \left( - \frac{b^{2}}{2r(1-r)} \right) db \right) dr. 
\end{align*}
But, for all $r \in (0,1)$, performing the change of variable $a = \frac{b^{2}}{2r(1-r)}$, we obtain
\begin{align*}
&\int_{0}^{\infty} b^{\delta-2} |\log(b)| \exp \left( - \frac{b^{2}}{2r(1-r)} \right) db \\
= &\left(r(1-r)\right)^{\frac{\delta-1}{2}} 2^{\frac{\delta-3}{2}}  \int_{0}^{\infty} a^{\frac{\delta-3}{2}} e^{-a} \, |\log(\sqrt{2r(1-r)a})| \, da \\
= &\left( r(1-r)\right)^{\frac{\delta-1}{2}} 2^{\frac{\delta-5}{2}} \left( \Gamma \left( \frac{\delta-1}{2} \right) |\log(2r(1-r))| + A\right)
\end{align*}
where $A:=   \int_{0}^{\infty} a^{\frac{\delta-3}{2}} e^{-a}  |\log(a)| \ da \in (0, +\infty) $, since  $\frac{\delta-3}{2} > -1$. Therefore
\begin{align*}
& \int_{0}^{1}\int_{0}^{\infty} \frac{p^{\delta}_{r}(b)}{b^{3}} |\mathcal{T}^{0}_{b} E^{\delta} [\Phi(X) | X_{r} = \cdot]| db \ dr \\
& \leq \frac{M L }{2^{3/2} \Gamma(\frac{\delta}{2})} \left\{  \Gamma \left( \frac{\delta-1}{2} \right) \int_{0}^{1}  \frac{|\log(2r(1-r))|}{\left(r(1-r)\right)^{1/2}} dr + A \int_{0}^{1} \frac{dr}{\left(r(1-r)\right)^{1/2}} \right\} 
\end{align*}
which is finite, whence the claim. We now prove that \eqref{bound_cond_exp} indeed holds. 
By Proposition \ref{additivity_pinned}, for all $r \in (0,1)$ and $\delta, x \geq 0$, denoting by $Z_{r}(\delta,x)$ a random variable in $L^{2}(0,1)$ distributed according to $Q^{\delta}\left( \cdot |X_{r} = x \right)$, we have 
\[  Z_{r}(\delta,x) \overset{(d)}{=} Z_{r}(\delta,0) + Z_{r}(0,x) \]
where $Z_{r}(\delta,0)$ and $Z_{r}(0,x)$ are two independent random variables with laws given respectively by $Q^{\delta}\left( \cdot |X_{r} = 0 \right)$ and $Q^{0}\left( \cdot |X_{r} = x \right)$. Therefore, for all functional $\Phi : L^{2}(0,1) \to \mathbb{R}$ satisfying \eqref{lipschitz_assumption}, for all $r \in (0,1)$ and $b > 0$, we have
\begin{align*}
E^{\delta} [\Phi(X) | X_{r} = b] &= Q^{\delta} \ [\Phi(\sqrt{X}) | X_{r} = b^{2}] \\ 
 &= \mathbb{E} \left[ \Phi \left( \sqrt{ Z_{r}(\delta,b^{2})} \right) \right] \\
 &= \mathbb{E} \left[ \Phi \left( \sqrt{Z_{r}(\delta,0) + Z_{r}(0,b^{2})} \right) \right].
\end{align*}
Hence
\begin{align*}
\left| \mathcal{T}^{0}_{b} E^{\delta} [\Phi(X) | X_{r} = \cdot] \right| &= |E^{\delta} [\Phi(X) | X_{r} = b] - E^{\delta} [\Phi(X) | X_{r} = 0] | \\
&= \left| \mathbb{E}  \left[ \Phi \left( \sqrt{Z_{r}(\delta,0) + Z_{r}(0,b^{2})} \right) - \Phi \left(  \sqrt{Z_{r}(\delta,0)} \right) \right] \right| \\
& \leq \mathbb{E}  \left[\left| \Phi \left( \sqrt{Z_{r}(\delta,0) + Z_{r}(0,b^{2})} \right) - \Phi \left(  \sqrt{Z_{r}(\delta,0)} \right) \right| \right].
\end{align*} 
But, by assumption \eqref{lipschitz_assumption}, we have
\[ \left| \Phi \left( \sqrt{Z_{r}(\delta,0) + Z_{r}(0,b^{2})} \right) - \Phi \left(  \sqrt{Z_{r}(\delta,0)} \right) \right| \leq L \left( \| Z_{r}(0,b^{2}) \|_{1} \right)^{1/2}. \]


Therefore
\begin{align*}
\left| \mathcal{T}^{0}_{b} E^{\delta} [\Phi(X) | X_{r} = \cdot] \right| &\leq L \, \mathbb{E}  \left[\| Z_{r}(0,b^{2})\|_{1}^{1/2} \right]\\
&= L \, E^{0} \left[ \ \| X^{2} \|_{1}^{1/2}  |X_{r} = b\right]\\
&= L \, E^{0} [ \ \|X\| \ |X_{r} = b],
\end{align*}
so there only remains to obtain a bound on $E^{0} [ \ \|X\| \ |X_{r} = b]$. To do so, we exploit the knowledge of the quantity $E^{0} [ \exp \left( - \lambda \|X\|^{2} \right)  |X_{r} = b]$, for all $ \lambda > 0$. Indeed, by equality (3.18) in \cite{EladAltman2019} we have
\[ E^{0} [ \exp \left( - \lambda \|X\|^{2} \right)  |X_{r} = b] = \exp \left[- C(r) \frac{b^{2}}{2} \right] \] 
where
\[ C(r) := \frac{\psi(1)}{\psi(r) \hat{\psi}(r)} - \frac{1}{r(1-r)}  , \]
with $\psi, \hat{\psi}$ associated, as in (3.13) and (3.14) of \cite{EladAltman2019}, to the function $\theta : [0,1] \to \mathbb{R}_{+}$ given by 
\[ \theta(u) = \lambda, \quad u \in [0,1]. \]
One finds easily the following expressions for $\psi$ and $\hat{\psi}$:
\[ \psi(u) = \frac{1}{\sqrt{2 \lambda}} \sinh \left( \sqrt{2 \lambda } u \right), \quad \hat{\psi} = \frac{1}{\sqrt{2 \lambda}} \sinh \left( \sqrt{2 \lambda} \left( 1-u \right) \right) \]
for all $u \in [0,1]$. In particular we obtain
\begin{align*}
 \frac{\psi(1)}{\psi(r) \hat{\psi}(r)} &= \frac{ \sqrt{2 \lambda} \sinh \left( \sqrt{2 \lambda } \right)}{\sinh \left( \sqrt{2 \lambda } r \right) \sinh \left( \sqrt{2 \lambda } (1- r) \right)} \\
&= \sqrt{2\lambda} \left( \text{coth} (\sqrt{2 \lambda} r) + \text{coth} (\sqrt{2 \lambda} (1-r)) \right),
\end{align*}
where $\text{coth}(x) := \frac{\cosh (x)}{\sinh (x)}$ for all $x \neq 0$ . Therefore, we have

\begin{align*}
C(r) &=  \sqrt{2 \lambda} \left(  \text{coth} (\sqrt{2 \lambda} r) + \text{coth} (\sqrt{2 \lambda} (1-r)) \right) - \frac{1}{r(1-r)} \\
&= \frac{1}{r} f(\sqrt{2\lambda}r) + \frac{1}{1-r} f \left(\sqrt{2 \lambda}(1-r)\right) 
\end{align*}

where $f : \mathbb{R} \to \mathbb{R}$ is defined by
\begin{equation*}
f(u)=
\begin{cases}
u \ \text{coth}(u)  - 1 , \quad &\text{if} \  u \neq 0 \\
0, \quad  &\text{if} \ u = 0.
\end{cases} 
\end{equation*} 
We thus obtain the expression
\begin{equation}
\label{expr_cond_exp}
E^{0} [ \exp \left( - \lambda \|X\|^{2} \right)  |X_{r} = b]  =  \exp \left[- \frac{b^{2}}{2} \left( \frac{1}{r} f(\sqrt{2\lambda}r) + \frac{1}{1-r} f\left(\sqrt{2 \lambda} \left(1-r\right) \right) \right) \right].
\end{equation}
There now remains to deduce from \eqref{expr_cond_exp} an expression for $E^{0} [ \|X\|  |X_{r} = b]$. To do so, we use the following lemma

\begin{lm}
Let $R$ be a nonnegative real variable such that $R > 0$ a.s. 
Then:
\[ \mathbb{E} [R] = \frac{1}{2\sqrt{\pi}} \int_{0}^{\infty} \lambda^{-3/2} \left( 1 -\mathbb{E} \left( \exp \left( - \lambda R^{2} \right) \right) \right) d\lambda. \]
\end{lm}

\begin{proof}
By Fubini-Tonnelli, we have
\begin{align*}
 \int_{0}^{\infty} \lambda^{-3/2} \left( 1 - \mathbb{E} \left( \exp \left( - \lambda R^{2} \right) \right) \right) d\lambda &= \mathbb{E} \left[  \int_{0}^{\infty} \lambda^{-3/2} \left( 1 - \exp \left(- \lambda R^{2} \right) \right) d\lambda \right] \\
&=  \mathbb{E} [R]  \int_{0}^{\infty} x^{-3/2} ( 1 - e^{-x}) dx, 
\end{align*}
where we performed the change of variable $x := R^{2} \lambda$ to obtain the last line (this is allowed, since $R>0$ a.s.). But, by Lemma 4.8 in \cite{EladAltman2019}, the last integral equals $- \Gamma \left(-\frac{1}{2} \right) = 2 \sqrt{\pi}$. The claim follows. 
\end{proof}                  

Applying this result to the random variable $ R:= \|X\| $ under the probability measure $ E^{0} [ \cdot  |X_{r} = b] $ over $L^{2}(0,1)$, we obtain
\begin{align*}
& E^{0} [ \ \|X\| \  |X_{r} = b] = \\
& \frac{1}{2\sqrt{\pi}} \int_{0}^{\infty} \lambda^{-3/2} \left( 1 -\exp \left[- \frac{b^{2}}{2} \left( \frac{1}{r} f(\sqrt{2\lambda}r) + \frac{1}{1-r} f \left(\sqrt{2 \lambda} \left(1-r\right) \right) \right) \right]
 \right) d\lambda.
\end{align*}
Performing the change of variable $x = \sqrt{2 \lambda}$, this yields
\begin{align*}
& E^{0} [ \ \|X\| \  |X_{r} = b] = \\
& \sqrt{\frac{2}{\pi}} \int_{0}^{\infty} x^{-2} \left( 1 -\exp \left[- \frac{b^{2}}{2} \left( \frac{1}{r} f(r x) + \frac{1}{1-r} f \left( \left(1-r \right) x\right) \right) \right]
 \right) d x,
\end{align*}
so it suffices to bound the latter integral. To do so, note that $f(u) = O(u^{2})$ when $u \to 0$, whereas $f(u) = O(u)$ as $u \to +\infty$. Hence there exists a universal constant $C > 0$ such that
\[ \forall u \geq 0, \quad f(u) \leq C \, u \wedge u^{2}. \] 
Therefore, recalling that $r \in (0,1)$, we have
\begin{align*}
\frac{1}{r} f(r x) + \frac{1}{1-r} f \left( \left(1-r) x \right) \right) &\leq \frac{1}{r} C (r x) \wedge (rx)^{2} + \frac{1}{1-r} C \left( (1-r) x\right) \wedge  \left( (1-r) x \right)^{2}  \\
&\leq 2 C \, x \wedge x^{2}.
\end{align*}

Hence, we have 

\begin{align*}
& \int_{0}^{\infty} x^{-2} \left( 1 -\exp \left[- \frac{b^{2}}{2} \left( \frac{1}{r} f(r x) + \frac{1}{1-r} f \left((1-r) x\right) \right) \right] \right) dx \\
&\leq  \int_{0}^{1} x^{-2} \left( 1 -\exp \left[- C b^{2} x^{2} \right] \right) d x + \int_{1}^{+\infty} x^{-2} \left( 1 -\exp \left[- C b^{2} x \right] \right) dx.
\end{align*}

The first integral is bounded by 
\[ \int_{0}^{1} x^{-2} C b^{2} x^{2} d x = C b^{2}, \]
while the second one is, by the change of variable $y=C b^2 \, x$, equal to
\begin{align*}
Cb^{2}\int_{Cb^{2}}^{+\infty} \frac{1}{y^{2}} \left(1- e^{-y} \right) dy &\leq Cb^{2} \left\{\mathbf{1}_{\{C b^2 \leq 1\}}  \left| \int_{Cb^{2}}^{1} \frac{1}{y^{2}} \left(1- e^{-y} \right) dy \right| + \int_{1}^{+\infty} \frac{1}{y^{2}} \left(1- e^{-y} \right) dy \right\} \\
&\leq Cb^{2}  \left\{ \mathbf{1}_{\{C b^2 \leq 1\}} \left| \int_{Cb^{2}}^{1} \frac{1}{y} dy \right| + \int_{1}^{+\infty} \frac{1}{y^{2}} dy \right\} \\
&= Cb^{2} \left(\left| \log \left( \frac{Cb^{2}}{2} \right) \right| + 1 \right).
\end{align*}
Thus, we obtain
\begin{align*}
 & \int_{0}^{\infty} x^{-2} \left( 1 -\exp \left[- \frac{b^{2}}{2} \left( \frac{1}{r} f(r x) + \frac{1}{1-r} f \left((1-r) x\right) \right) \right] \right) dx \\
&\leq C b^{2} \left( \frac{1}{2} \left| \log \left( \frac{Cb^{2}}{2} \right) \right| + 1 \right) \\
& \leq C' b^{2} |\log(b)|,
\end{align*}
where $C'$ is some universal constant. Setting $M:= \sqrt{\frac{2}{\pi}}C'$, the claim follows.

\end{proof}

\subsection{Differentiability properties of conditional expectations}

In this section, we aim at proving that, for any $\delta \geq 0$, for a large class of functionals $\Phi:L^{2}(0,1) \to \mathbb{R}$, the quantity $E^{\delta} [\Phi(X) | X_{r} = b]$
is twice differentiable in $b$ at $b=0$. To do so we shall exploit Proposition \ref{levy_meas_pinned_sqred_bridges} above, which provides the existence, for all $r \in (0,1)$, of a L\'evy measure on $C_{+}([0,1])$ corresponding to the convolution semi-group $\left( Q^{0} \left[ \ \cdot \ | X_{r} = x \right] \right)_{x \geq 0}$. Note that the measures $M^{r}$, $r \in (0,1)$, are not finite. However they have the following important property:

\begin{lm}
\label{finite_moment_nr}
For all $r \in (0,1)$,
\[ \int \| X \|_{1} \ dM^{r}(X) = \frac{1}{3} < \infty. \]
\end{lm}

\begin{proof}
For all $x \geq 0$ and $\lambda > 0$, by \eqref{Levy_meas}, we have
\begin{align*}
Q^{0} \left[ \exp \left( - \lambda \| X \|_{1} \right) | X_{r} =x \right] =  \exp \left( - x \int \left( 1 - \exp \left( - \lambda \| X \|_{1} \right) \right) dM^{r}(X) \right).
\end{align*}
On the other hand, by \eqref{expr_cond_exp}, we have
\begin{align*}
 Q^{0} \left[ \exp \left( - \lambda \| X \|_{1} \right) | X_{r} =x \right]  &=  E^{0} \left[ \exp \left( - \lambda \| X \|^{2}  \right) | X_{r} =  \sqrt{x} \right] \\
 &=  \exp \left[- \frac{x}{2} \left( \frac{1}{r} f(\sqrt{2\lambda}r) + \frac{1}{1-r} f \left(\sqrt{2 \lambda} \left(1-r)\right) \right) \right) \right].
\end{align*}
Therefore, we deduce that, for all $\lambda > 0$,
\begin{equation}
\label{expr_moment_nr}
 \int \left( 1 - \exp \left( - \lambda \| X \|_{1} \right) \right) dM^{r}(X) = \frac{1}{2} \left( \frac{1}{r} f(\sqrt{2\lambda}r) + \frac{1}{1-r} f \left(\sqrt{2 \lambda} (1-r)\right) \right). 
\end{equation} 
But, by the monotone convergence theorem, we have
\[ \lim_{\lambda \to 0} \frac{1}{\lambda}   \int \left( 1 - \exp \left( - \lambda \| X \|_{1} \right) \right) dM^{r}(X)  =  \int \| X \|_{1} \ dM^{r}(X). \]
On the other hand, since $f(x) = \frac{x^{2}}{3} + o(x^{2})$ as $x \to 0$, we have
\[  \lim_{\lambda \to 0} \frac{1}{\lambda} \left( \frac{1}{r} f(\sqrt{2\lambda}r) + \frac{1}{1-r} f \left(\sqrt{2 \lambda} (1-r)\right) \right) = \frac{2}{3}. \]
Therefore, dividing both sides of \eqref{expr_moment_nr} by $\lambda$ and taking the limit $\lambda \to 0$, we obtain
\[ \int \| X \|_{1} \ dM^{r}(X) = \frac{1}{3},\]
which yields the claim.
\end{proof}

We are now in position to establish Proposition \ref{statement_diff_prop}, which is an immediate consequence of the following result. 

\begin{prop}
\label{diff_prop}
Let $\delta \geq 0$, and let $\Phi$ be a functional on $L^{2}(0,1)$ of the form
\begin{equation}
\label{lip_func_of_square}
 \Phi(X) = \Psi(X^{2}), \quad X \in L^{2}(0,1)
\end{equation}
where $\Psi: L^{1}(0,1) \to\mathbb{R}$ is bounded and globally Lipschitz continuous. Then, for all $r \in (0,1)$ and $b \geq 0$, we have
\begin{align}
\label{integral_formula_for_cond_exp}
&E^{\delta} \left[ \Phi(X) | X_{r} =b \right] = E^{\delta} \left[ \Phi(X) | X_{r} =0 \right] \\
\nonumber &+ 2 \int_{0}^{b} a \int \left(  E^{\delta} \left[ \Phi \left( \sqrt{ X^{2} + Z} \right) | X_{r} =a \right] - E^{\delta} \left[ \Phi(X) | X_{r} =a \right] \right) dM^{r}(Z) da.
\end{align} 
In particular, the quantity $E^{\delta} [\Phi(X) | X_{r} = b]$
is twice differentiable in $b$ at $0$, and 
\begin{align*}
&\frac{d^{2}}{db^{2}}  E^{\delta} [\Phi(X) | X_{r} = b] \biggr\rvert_{b=0}  = \\
&2 \int \left(  E^{\delta} \left[ \Phi \left(\sqrt{ X^{2} + Z} \right) | X_{r} =0 \right] - E^{\delta} \left[ \Phi(X) | X_{r} =0 \right] \right) dM^{r}(Z).
\end{align*}
\end{prop}

\begin{rk}
The idea behind this Proposition is the fact that for all $r \in (0,1)$
\[ \left( Q^{0} \left[\ \cdot \ | X_{r} =x \right] \right)_{x \geq 0} \]
is a convolution semi-group, to which one could, using the same techniques as in \cite{pitman1982decomposition}, associate a subordinator with values in $C_{+}([0,1])$. That subordinator would be a compound Poisson point process with intensity $dt \otimes M^{r}$. For such a process one should have an It\^{o} formula as in Theorem 5.1 of \cite{ikeda2014stochastic}, from which formula \eqref{integral_formula_for_cond_exp} would then follow simply by taking expectations. Although such a strategy should be possible to implement using the constructions done in \cite{pitman1982decomposition}, since we do not need any pathwise statement, we prefer to resort to a more basic proof based on a density argument.   
\end{rk}

\begin{proof}
The second statement follows from equality \eqref{integral_formula_for_cond_exp}. Indeed, for all fixed $Z \in C_{+}([0,1])$, the quantity 
\[E^{\delta} \left[ \Phi \left(\sqrt{ X^{2} + Z} \right) | X_{r} =a \right] - E^{\delta} \left[ \Phi(X) | X_{r} =a \right]  \]
is continuous in $a$. Moreover, it is dominated by $L \|Z\|_{1}$, where  $L>0$ is a Lipschitz constant for $\Psi$. Since $\|Z\|_{1}$ is integrable w.r.t. $M^{r}(dZ)$, we deduce that the quantity
\[F(a) := \int \left( E^{\delta} \left[ \Phi \left(\sqrt{ X^{2} + Z} \right) | X_{r} =a \right] - E^{\delta} \left[ \Phi(X) | X_{r} =a \right] \right) M^{r}(dZ) \]
is continuous in $a$. But, by \eqref{integral_formula_for_cond_exp}, we have, for all $b \geq 0$,
\[ E^{\delta} \left[ \Phi(X) | X_{r} =b \right] = E^{\delta} \left[ \Phi(X) | X_{r} =0 \right] + 2 \int_{0}^{b} a F(a) da.  \]
Hence, we deduce that $b \to E^{\delta} \left[ \Phi(X) | X_{r} =b \right] $ is twice differentiable at 0, with its derivative there given by $2 F(0)$. This yields the second statement.

We now prove the first statement. We start by proving \eqref{integral_formula_for_cond_exp} for all $\Phi \in \mathscr{S}$. By linearity, we may assume that $\Phi$ is of the form
\eqref{exp_functional}, which is tantamount to $\Phi$ satisfying \eqref{lip_func_of_square}, with $\Psi$ given by
\[ \Psi(Z) = \exp \left( - \langle \theta, Z \rangle \right), \quad Z \in L^{1}(0,1)\]
for some $\theta : [0,1] \to \mathbb{R}_{+}$ bounded and Borel. To do so, note that, as a consequence of Proposition \ref{additivity_pinned}, for all $x \geq 0$, we have
\[ Q^{\delta} [\Psi(X) | X_{r} = x] = Q^{\delta} [\Psi(X) | X_{r} = 0] \, Q^{0} [\Psi(X) | X_{r} = x], \]
so that, by  \eqref{Levy_meas},
\begin{align*}
Q^{\delta} [\Psi(X) | X_{r} = x] = Q^{\delta} [\Psi(X) | X_{r} = 0] \, \exp \left(x \int \left( \Psi(Z) - 1 \right) dM^{r}(Z) \right). 
\end{align*}
Hence, differentiating in $x$, we have
\begin{align*}
\frac{d}{dx}  Q^{\delta} [\Psi(X) | X_{r} = x] &= \int \left( \Psi(Z) - 1 \right) dM^{r}(Z) \, Q^{\delta} [\Psi(X) | X_{r} = x] \\  
&= \int \left(Q^{\delta} [\Psi(Z) \Psi(X) | X_{r} = x]  - Q^{\delta} [\Psi(X) | X_{r} = x]  \right) dM^{r}(Z) \\
&= \int \left(Q^{\delta} [\Psi(X + Z) | X_{r} = x]  - Q^{\delta} [\Psi(X) | X_{r} = x]  \right) dM^{r}(Z).
\end{align*}
Hence, for all $x \geq 0$, we have
\begin{align*}
Q^{\delta} [\Psi(X) | X_{r} = x] &= Q^{\delta} [\Psi(X) | X_{r} = 0] \\
&+ \int_{0}^{x}  \int \left(Q^{\delta} [\Psi(X + Z) | X_{r} = y]  - Q^{\delta} [\Psi(X) | X_{r} = y]  \right) dM^{r}(Z) dy. 
\end{align*}
Therefore, for all $b \geq 0$, we have 
\begin{align*}
&E^{\delta} [\Phi(X) | X_{r} = b^{2}] = E^{\delta} [\Phi(X) | X_{r} = 0] \\
&+ \int_{0}^{b^{2}}  \int \left(E^{\delta} [\Phi(\sqrt{X + Z}) | X_{r} = \sqrt{y}]  - E^{\delta} [\Phi(X) | X_{r} = \sqrt{y}]  \right) dM^{r}(Z) dy, 
\end{align*}
so that, performing the change of variable $a := \sqrt{y}$, we obtain \eqref{integral_formula_for_cond_exp}. Let now $\Phi$ be of the form \eqref{lip_func_of_square}, with $\Psi: L^{1}(0,1) \to \mathbb{R}$ bounded and globally Lipschitz continuous. We can construct a family of approximating functionals $(\Phi^{d}_{n,k})_{d,n,k \geq 1}$ as in the proof of Proposition \ref{density_l1}. Although $\Psi$ is not necessarily $C^{1}$, reasoning as in the proof of Proposition \ref{density_l1}, we can check that these sequences will satisfy
\[ \underset{d \to \infty}{\lim} \ \underset{n \to \infty}{\lim} \ \underset{k \to \infty}{\lim} \Phi^{d}_{n,k} = \Phi, \]
where each limit happens \textit{almost} in the sense of $(A_1)$ : conditions \ref{fst_condition} and \ref{thd_condition} of Definition \ref{pdi_conv} hold, only condition \ref{snd_condition} may not hold. Since, for all $d,n,k \geq 1$, $\Phi^{d}_{n,k}$ lies in $\mathscr{S}$, by the previous point, we have
\begin{align}
\label{integral_formula_for_cond_exp_nk}
& E^{\delta} \left[ \Phi^{d}_{n,k}(X) | X_{r} =b \right] = E^{\delta} \left[ \Phi^{d}_{n,k}(X) | X_{r} =0 \right] \\
\nonumber &+ 2 \int_{0}^{b} a \int \left(  E^{\delta} \left[ \Phi^{d}_{n,k} \left( \sqrt{ X^{2} + Z} \right) | X_{r} =a \right] - E^{\delta} \left[ \Phi^{d}_{n,k}(X) | X_{r} =a \right] \right) dM^{r}(Z) da.
\end{align} 
Now, for all $X \in L^{2}(0,1)$, we have
\[ \underset{d,n,k \to \infty}{\lim} \Phi^{d}_{n,k}(X) = \Phi(X), \]
with the domination 
\[ \forall d,n,k \geq 1, \quad \| \Phi^{d}_{n,k}\|_{\infty} \leq \| \Phi \|_{\infty}, \]
from which we deduce that
\[ \underset{d,n,k \to \infty}{\lim} E^{\delta} [\Phi^{d}_{n,k}(X) | X_{r} = b] = E^{\delta} [\Phi(X) | X_{r} = b], \]
and
\[ \underset{d,n,k \to \infty}{\lim} E^{\delta} [\Phi^{d}_{n,k}(X) | X_{r} = 0] = E^{\delta} [\Phi(X) | X_{r} = 0]. \]
We also deduce therefrom that, for all $Z \in C_{+}([0,1])$ and $a \in [0,b]$, we have
\begin{align*} 
\underset{d,n,k \to \infty}{\lim} &E^{\delta} \left[ \Phi^{d}_{n,k} \left( \sqrt{ X^{2} + Z} \right) | X_{r} =a \right] - E^{\delta} \left[ \Phi^{d}_{n,k}(X) | X_{r} =a \right] = \\
&E^{\delta} \left[ \Phi \left( \sqrt{ X^{2} + Z} \right) | X_{r} =a \right] - E^{\delta} \left[ \Phi(X) | X_{r} =a \right],
\end{align*}
and, by condition \ref{thd_condition} in Definition \ref{pdi_conv}, these three limits happen with uniform domination by $\|Z\|_{1}$. Since $\|Z\|_{1}$ is integrable with respect to $dM^{r}(Z) $ over $C_{+}([0,1])$, by three successive applications of the dominated convergence theorem, we deduce that
\begin{align*} 
\underset{d,n,k \to \infty}{\lim} &\int_{0}^{b} a \int \left( E^{\delta} \left[ \Phi^{d}_{n,k} \left( \sqrt{ X^{2} + Z} \right) | X_{r} =a \right] - E^{\delta} \left[ \Phi^{d}_{n,k}(X) | X_{r} =a \right] \right) dM^{r}(Z) da  \\
&=\int_{0}^{b} a \int \left(E^{\delta} \left[ \Phi \left( \sqrt{ X^{2} + Z} \right) | X_{r} =a \right] - E^{\delta} \left[ \Phi(X) | X_{r} =a \right] \right) d M^{r}(Z) da. 
\end{align*}
Hence, sending successively $k, n$ and $d$ to $ \infty$ in \eqref{integral_formula_for_cond_exp_nk}, we deduce that $\Phi$ also satisfies \eqref{integral_formula_for_cond_exp}. 
This yields the claim.
\end{proof}

As a consequence of the above proposition, we deduce an improved order $0$ estimate for functionals of the form \eqref{lip_func_of_square}.

\begin{prop}
\label{bound_lip_func_square}
Let $\Phi: L^{2}(0,1) \to \mathbb{R}$ be a functional of the form \eqref{lip_func_of_square}, with $\Psi:L^{1}(0,1) \to \mathbb{R}$ bounded and globally Lipschitz continous, with Lipschitz constant $L>0$. Then, for all $\delta \geq 0$ the following holds
\begin{equation}
\label{bound_cond_exp_for_func_of_square}
\forall r \in (0,1),  \forall b > 0, \quad \left| \mathcal{T}^{0}_{b} E^{\delta} [\Phi(X) | X_{r} = \cdot] \right| \leq \frac{L}{3} \ b^{2}.
\end{equation} 
In particular, for all $\delta > 1$, the function 
\[(r,b) \mapsto  \mathcal{T}^{0}_{b} E^{\delta} [\Phi(X) | X_{r} = \cdot] \]
 is integrable with respect to the measure $\frac{p^{\delta}_{r}(b)}{b^{3}} \, dr \, db$ on $(0,1) \times \mathbb{R}^{*}_{+}$.
\end{prop}

\begin{proof}
The integrability claim is deduced from the estimate \eqref{bound_cond_exp_for_func_of_square} using similar computations as in the proof of Proposition \ref{bound_lip_func} above, so we are only left to prove the bound \eqref{bound_cond_exp_for_func_of_square}.
By \eqref{integral_formula_for_cond_exp}, for all $\delta \geq 0$, $r \in (0,1)$ and $b \geq 0$, we have
\begin{align*}
&\left| \mathcal{T}^{0}_{b} E^{\delta} [\Phi(X) | X_{r} = \cdot] \right| = |E^{\delta} [\Phi(X) | X_{r} = b] - E^{\delta} [\Phi(X) | X_{r} = 0] | \\
&\leq 2 \int_{0}^{b} a \int \left|  E^{\delta} \left[ \Phi \left( \sqrt{ X^{2} + Z} \right) | X_{r} =a \right] - E^{\delta} \left[ \Phi(X) | X_{r} =a \right] \right| dM^{r}(Z) da \\
&\leq 2 \int_{0}^{b} a \int L \|Z\|_{1}  dM^{r}(Z) da.
\end{align*}
But, by Lemma \ref{finite_moment_nr}, the last expression equals $  2 \int_{0}^{b} a \, \frac{L}{3} \, da = \frac{L}{3}b^{2}$, whence the claim.
\end{proof}

\subsection{A second-order Taylor estimate}

\begin{prop}
\label{taylor_est_cond_exp}
Let $\delta > 0$, and let $\Phi$ be a functional on $L^{2}(0,1)$ of the form
\[ \Phi(X) = \Psi(X^{2}), \quad X \in L^{2}(0,1) \]
where $\Psi : L^{1}(0,1) \to \mathbb{R} $ is bounded and globally Lipschitz continuous. Assume furthermore that there exists $L>0$ such that 
\[ \forall X, Z, Z' \in L^{1}(0,1), \quad |\Psi(X+Z+Z') - \Psi(X+Z) - \Psi(X+Z') + \Psi(X) | \leq L \|Z\|_{1} \|Z'\|_{1}. \]
Then, for all $r \in (0,1)$, the quantity $E^{\delta} [\Phi(X) | X_{r} = b] $
is twice differentiable in $b$ at $0$. Moreover, for all $b \geq 0$,
\begin{equation}
\label{the_estimate}
\left| \mathcal{T}^{2}_{0,b} E^{\delta} [\Phi(X) | X_{r} = \cdot] \right| \leq  L b^{4}.
\end{equation} 
In particular, the function 
\[(r,b) \mapsto  \mathcal{T}^{2}_{0,b} E^{\delta} [\Phi(X) | X_{r} = \cdot]  \]
 is integrable with respect to the measure $\frac{p^{\delta}_{r}(b)}{b^{3}} \, dr \, db$ on $(0,1) \times \mathbb{R}^{*}_{+}$.
\end{prop}

\begin{rk}
This proposition applies in particular to any $\Phi \in \mathcal{S} C^{1,1}_{b}(L^{1}(0,1))$.
\end{rk}

\begin{proof}
The differentiability property follows  from Proposition \ref{diff_prop}.  Moreover, the integrability claim follows from the estimate \eqref{the_estimate} using similar computations as in the proof of Proposition \ref{bound_lip_func}.
%
So there only remains to prove \eqref{the_estimate}. To do so, remark that, for all $b \geq 0$, we have
\[ E^{\delta} [\Phi(X) | X_{r} = b] = G(b^{2}), \]
where, for all $x \geq 0$, $G(x) := Q^{\delta} [\Psi(X) | X_{r} = x]$. As a consequence, we have
\begin{align*}
\mathcal{T}^{2}_{0,b} E^{\delta} [\Phi(X) | X_{r} = \cdot] =  G(b^{2}) - G(0) - b^{2} G'(0).
\end{align*}
Our claim will then follow from Taylor's theorem, once we have proved that $G$ is $C^{2}$ on $\mathbb{R}_{+}$. To do so, note that, by equality $\eqref{integral_formula_for_cond_exp}$, for all $x \geq 0$, we have
\[ G(x) = G(0) + \int_{0}^{x} \int \left( Q^{\delta} [\Psi(X+Z) | X_{r} = y] - Q^{\delta} [\Psi(X) | X_{r} = y] \right) d M^{r}(Z) dy. \]
Now, by the Lipschitz property of $\Psi$, and since $\int \|Z\|_{1} M^{r}(dZ) < \infty$, the function
\begin{align*}
\begin{cases} 
\mathbb{R}_{+} &\to \mathbb{R} \\
y &\mapsto \int \left( Q^{\delta} \left[ \Psi \left(X + Z \right) | X_{r} =y \right] - Q^{\delta} \left[ \Psi(X) | X_{r} =y \right] \right) M^{r}(dZ)
\end{cases}
\end{align*}
is continuous. Therefore, $G$ is differentiable on $\mathbb{R}_{+}$, and
\begin{equation}
\label{der_of_F}
 G'(x) = \int \left( Q^{\delta} [\Psi(X+Z) | X_{r} = x] - Q^{\delta} [\Psi(X) | X_{r} = x] \right) d M^{r}(Z), \quad x \geq 0.
\end{equation}
By the same arguments, for all $Z \in C_{+}([0,1])$, the quantity 
\[ Q^{\delta} [\Psi(X+Z) | X_{r} = x] -  Q^{\delta} [\Psi(X) | X_{r} = x]  \]
is differentiable on $\mathbb{R}_{+}$ with respect to $x$, with derivative given by
\begin{align*}
&\int \left( Q^{\delta} [\Psi(X+Z+Z') | X_{r} = x] - Q^{\delta} [\Psi(X+Z) | X_{r} = x] \right) dM^{r}(Z') \\
&- \int \left( Q^{\delta} [\Psi(X+Z') | X_{r} = x] - Q^{\delta} [\Psi(X) | X_{r} = x] \right) d M^{r}(Z') \\
&= \int Q^{\delta} \left[ \Psi(X+Z+Z') - \Psi(X+Z) - \Psi(X+Z') + \Psi(X) | X_{r} = x \right] d M^{r}(Z').
\end{align*}
Since for all $X, Z' \in C_{+}([0,1])$ we have
\[ | \Psi(X+Z+Z') - \Psi(X+Z) - \Psi(X+Z') + \Psi(X) | \leq L \|Z\|_{1} \|Z'\|_{1}, \]
we deduce that
\begin{align*}
\left| \frac{d}{dx} \left( Q^{\delta} [\Psi(X+Z) | X_{r} = x] -  Q^{\delta} [\Psi(X) | X_{r} = x] \right) \right| &\leq L \|Z\|_{1} \int \|Z'\|_{1} dM^{r}(Z') \\
&= \frac{L}{3} \|Z\|_{1}. 
\end{align*}
Since $\int \|Z\|_{1} dM^{r}(Z) < \infty$, we deduce that $G'$ is differentiable on $\mathbb{R}_{+}$, with derivative given for all $x \geq 0$ by $G''(x) = \int \int F(Z,Z') \, d M^{r}(Z') \, dM^{r}(Z)$, where
\begin{align*}
F(Z,Z') := Q^{\delta} [\Psi(X+Z+Z') - \Psi(X+Z) 
- \Psi(X+Z') + \Psi(X) \, | X_{r} = x].
\end{align*}
Note furthermore that
\[ \| G'' \|_{\infty} \leq L \int \|Z\|_{1} dM^{r}(Z) \int \|Z'\|_{1} dM^{r}(Z') = \frac{L}{9}. \]
Hence, by Taylor's theorem, we have, for all $x \geq 0$,
\begin{align*} 
|G(x) - G(0) - x G'(0)| \leq \| G'' \|_{\infty} \frac{x^{2}}{2} \leq  \frac{L}{18} x^{2}.  
\end{align*}
Therefore, for all $b \geq 0$, we have
\begin{align*}
|G(b^{2}) - G(0) - b^{2} G'(0)| &\leq \frac{L}{18} b^{4} \leq L \, b^{4}.
\end{align*}
This yields the claimed estimate.
\end{proof}

\begin{rk}
\label{open_quest_estimates}
Propositions \ref{diff_prop} and \ref{taylor_est_cond_exp} above a priori apply for functionals $\Phi$ of the form
\[ \Phi(X) = \Psi(X^{2}), \quad X \in L^{2}(0,1), \]
with $\Psi : L^{1}(0,1) \to \mathbb{R}$ sufficiently regular. It is not clear whether one could relax these conditions. For example, it is an open question whether the estimates would still hold for any $\Phi \in C^{1}_{b}(L^{2}(0,1))$, as is the case for the Taylor estimate at order $0$ obtained in Proposition \ref{bound_lip_func}. If such were the case, then we could also relax the conditions on $\Phi$ in Theorems \ref{general_IbPF_1} and \ref{general_IbPF_01}.
\end{rk}

\section{Extension of the integration by parts formulae to general functionals}

\label{section_ibpf}

We now turn to the proof of Theorems \ref{general_IbPF_13}, \ref{general_IbPF_1}, and \ref{general_IbPF_01}, stating that the IbPF on $P^{\delta}$ for $\delta \in (0,3)$ extend to general, sufficiently regular functionals on $L^{2}(0,1)$. To do so, we will use the density results stated in section \ref{section_density} to approximate  a general functional by elements of $\mathscr{S}$. Then we shall use the estimates obtained in \ref{section_ibpf} to show that the last term appearing in the IbPF converges when we take such approximating sequences. A little caveat here lies in the fact that our estimates concern Taylor remainders of the conditional expectations $E^{\delta} [\Phi(X) | X_{r} = b]$ for $b$ near $0$ and $r \in (0,1)$, 
while the last term in the IbPF contains Taylor remainders of the quantities $\Sigma^{\delta}_{r}(\Phi (X) \,|\, \cdot\,)$.
However, since the latter differs from the former only by a smooth function of $b^{2}$, we can actually re-express Taylor remainders of the latter as the sum of Taylor remainders of the former and some additional nicely-behaved terms. More precisely, the following holds:

\begin{lm}
Let $h \in C^{2}_{c}(0,1)$. Then, for all $\delta \in (1,3)$ and  $\Phi \in C^{1}_{b}(L^{2}(0,1))$  we have
\begin{equation}
\label{reexpressed_term_13}
\begin{split} 
&\int_{0}^{1}  h_{r}  \int_0^\infty b^{\delta-4} \Big[ \mathcal{T}^{\,0}_{b} \, \Sigma^{\delta}_{r}(\Phi (X) \,|\, \cdot\,) \Big]
  \d b \d r \\
= &\int_{0}^{1} dr h(r) \int_{0}^{\infty} db \ p^{\delta}_{r}(b) \frac{1}{b^{3}} \mathcal{T}^{0}_{b} E^{\delta} [\Phi(X) | X_{r} = \cdot] \\
+ &\frac{\Gamma(\frac{\delta-3}{2})}{\Gamma(\frac{\delta}{2})} \int_{0}^{1} dr \frac{h(r)}{(2r(1-r))^{3/2}} E^{\delta} [\Phi(X) | X_{r} = 0].
\end{split}
\end{equation}
Moreover, for all  $\Phi \in \mathcal{S}C^{1}_{b}\left( L^{1}(0,1) \right)$, we have
\begin{equation}
\label{reexpressed_term_1}
\begin{split} 
&\frac{1}{4} \int_{0}^{1} \d r \, h_r\, \frac{{\rm d}^{2}}{{\rm d} b^{2}} \, \Sigma^{1}_{r} [\Phi(X) \, | \, b]  \biggr\rvert_{b=0} \\
= - &\frac{1}{2\sqrt{2 \pi}} \int_{0}^{1} dr \frac{h(r)}{(r(1-r))^{3/2}} E^{1} [\Phi(X) | X_{r} = 0]  \\
+ &\frac{1}{2\sqrt{2 \pi}} \int_{0}^{1} dr \frac{h(r)}{(r(1-r))^{1/2}} \frac{d^{2}}{db^{2}} E^{1} [\Phi(X) | X_{r} = b] \biggr\rvert_{b=0}. 
\end{split}
\end{equation}
Finally, for all $\delta \in (0,1)$ and $\Phi \in \mathcal{S}C^{1,1}_{b}\left( L^{1}(0,1) \right)$, we have
\begin{equation}
\label{reexpressed_term_01}
\begin{split} 
&\int_{0}^{1} h_{r}  \int_0^\infty b^{\delta-4} \Big[ \mathcal{T}^{\,-2}_{b} \, \Sigma^{\delta}_{r}(\Phi (X) \,|\, \cdot\,) \Big]
  \d b \d r = \\
&\int_{0}^{1} dr \ h(r)  \int_{0}^{\infty} db \ \frac{p^{\delta}_{r}(b)}{b^{3}} \mathcal{T}^{-2}_{b} \left( E^{\delta} [\Phi(X) | X_{r} = \cdot] \right) \\
&+ \sum_{0 \leq j \leq 1} \frac{\Gamma \left(\frac{\delta-3}{2} + j \right)}{\Gamma \left(\frac{\delta}{2} \right)} \int_{0}^{1} dr \frac{h(r)}{2^{3/2} (r(1-r))^{3/2-j}} \frac{d^{2j}}{dx^{2j}} E^{\delta} [\Phi(X) | X_{r} = x]\biggr\rvert_{x=0}.
\end{split}
\end{equation}

\end{lm}

\begin{proof}
We prove only the equality for $\delta \in (1,3)$, since the other cases can be treated in the same way. Then, for $h \in C^{2}_{c}(0,1)$, $\Phi \in C^{1}_{b}(L^{2}(0,1))$, we have
\[\begin{split}
&\int_{0}^{1}  h_{r}  \int_0^\infty b^{\delta-4} \Big[ \mathcal{T}^{\,0}_{b} \, \Sigma^{\delta}_{r}(\Phi (X) \,|\, \cdot\,) \Big]
  \d b \d r \\
&= \int_{0}^{1} dr h(r) \int_{0}^{\infty} db \ p^{\delta}_{r}(b) \frac{1}{b^{3}} (E^{\delta} [\Phi(X) | X_{r} = b] - E^{\delta} [\Phi(X) | X_{r} = 0]) \\
&+  \int_{0}^{1} dr h(r) \int_{0}^{\infty} db \ b^{\delta-4}  \left(\gamma(r,b) - \gamma(r,0) \right) E^{\delta} [\Phi(X) | X_{r} = 0], 
\end{split} \]
where, for $r \in (0,1)$ and $b \geq 0$, we have set $\gamma(r,b) := \frac{p^{\delta}_{r}(b)}{b^{\delta-1}}$. Note that, in the first integral in the right-hand side, by \eqref{density} and Prop. \ref{bound_lip_func}, the integrand is of order $O(b^{\delta-2} \log(b))$ when $b \to 0$, and displays exponential decays as $b \to \infty$, so is integrable. Moreover, recall  from \eqref{density} that
\[ \gamma(r,b) = \frac{1}{2^{\frac{\delta}{2} -1} (r(1-r))^{\delta/2} \Gamma(\frac{\delta}{2})} \exp \left(- \frac{b^{2}}{2r(1-r)} \right), \]
so the integrand in the second integral in the right-hand side is of order $O(b^{\delta-2})$ when $b \to 0$ and decays exponentially as $b \to \infty$, so is also integrable. As a consequence, the integral in the left-hand side is absolutely convergent as well. Moreover, for all $r \in (0,1)$, applying equality (4.12) in \cite{EladAltman2019} (with $x=\frac{\delta-3}{2}$ and $C=\frac{1}{r(1-r)}$), we have
\[ \begin{split}
&\int_{0}^{\infty} db \ b^{\delta-4}  \left(\gamma(r,b) - \gamma(r,0) \right) \\
&= \frac{1}{2^{\frac{\delta}{2} -1} (r(1-r))^{\delta/2} \Gamma(\frac{\delta}{2})} \int_{0}^{\infty} db \ b^{\delta-4}  \left(\exp \left(-\frac{b^{2}}{2 r(1-r)} \right) - 1 \right) \\
&= \frac{\Gamma(\frac{\delta-3}{2})}{\Gamma(\frac{\delta}{2})} \frac{1}{(2r(1-r))^{3/2}}.
\end{split} \]
We thus obtain the claim.
\end{proof}

\subsection{Extension of the IbPF for $\delta \in (1,3)$}


\begin{proof}[Proof of Theorem \ref{general_IbPF_13}]
Given $\Phi \in C^{1}_{b}( L^{2}(0,1))$, consider $(\Phi^{m, d}_{n,k})_{m,d,n,k \geq 1}$ approximating $\Phi$ as in Proposition \ref{density_l2}. Then, for all $m,d,n,k \geq 1$, $\Phi^{m, d}_{n,k} \in \mathscr{S}$. Hence, by \eqref{onetothree} combined with \eqref{reexpressed_term_13}, we have
\begin{align}
\label{general_formula_13_mdnk}
E^{\delta} (\partial_{h} \Phi^{m, d}_{n,k}(X) ) &= - E^{\delta} (\langle h '' , X \rangle\Phi^{m, d}_{n,k}(X) ) \\
\nonumber& - \kappa(\delta) \int_{0}^{1} dr h(r) \int_{0}^{\infty} db \ p^{\delta}_{r}(b) \frac{1}{b^{3}} \mathcal{T}^{0}_{b} E^{\delta} [\Phi^{m, d}_{n,k}(X) | X_{r} = \cdot] \\
\nonumber& - \frac{\kappa(\delta) \Gamma(\frac{\delta-3}{2})}{\Gamma(\frac{\delta}{2})} \int_{0}^{1} dr \frac{h(r)}{(2r(1-r))^{3/2}} E^{\delta} [\Phi^{m, d}_{n,k}(X) | X_{r} = 0].
\end{align}
Hence, to obtain the claim, it suffices to show that, as we send $k,n,d$ and $m$ to $+\infty$, each term appearing in \eqref{general_formula_13_mdnk} converges to the same term with $\Phi^{m, d}_{n,k}$ replaced with $\Phi$. Here, the convergence \eqref{conv_approx_sequences} comes into play. Indeed, as a consequence of condition \ref{snd_condition} in Definition \ref{pdi_conv}, and since $\|h\|_{\infty} (1+ \|X\|)$ is integrable w.r.t. $P^{\delta}$, by the dominated convergence theorem, we have
\[ \underset{m,d,n,k \to \infty}{\lim} E^{\delta} (\partial_{h} \Phi^{m, d}_{n,k}(X) ) = E^{\delta} (\partial_{h} \Phi(X) ), \]
where we take the limits $k,n,d$ and $m$ successively. Moreover, by the condition \ref{fst_condition}, and since $|\langle h'', X \rangle| \leq \|h''\|_{\infty} \|X\|$ is integrable with respect to $P^{\delta}$, by dominated convergence, we have
\[  \underset{m,d,n,k \to \infty}{\lim} E^{\delta} (\langle h '' , X \rangle\Phi^{m, d}_{n,k}(X) ) = E^{\delta} (\langle h '' , X \rangle\Phi(X) ). \]
In a similar way, we obtain that
\[ \begin{split}
\underset{m,d,n,k \to \infty}{\lim}  &\int_{0}^{1} dr \frac{h(r)}{(2r(1-r))^{3/2}} E^{\delta} [\Phi^{m, d}_{n,k}(X) | X_{r} = 0] \\
=  &\int_{0}^{1} dr \frac{h(r)}{(2r(1-r))^{3/2}} E^{\delta} [\Phi(X) | X_{r} = 0]. 
\end{split}\]
Finally, for all $r \in (0,1)$, $b>0$ and $X,Y\in C_{+}([0,1])$, by dominated convergence, we have
\[  \underset{m,d,n,k \to \infty}{\lim} \mathcal{T}^{0}_{b} E^{\delta} [\Phi^{m, d}_{n,k}(X) | X_{r} = \cdot] = \mathcal{T}^{0}_{b} E^{\delta} [\Phi(X) | X_{r} = \cdot].\] 
Moreover, as a consequence of condition \ref{thd_condition} in Definition \ref{pdi_conv}, and by Propositions \ref{bound_lip_func} and \ref{bound_lip_func_square}, these convergences all happen with uniform domination by the function $(r,b) \mapsto b^{2}(|\log(b)| +1)$. Since the latter is integrable w.r.t.  $p^{\delta}_{r}(b) \frac{1}{b^{3}} \, dr \, db$ on $(0,1) \times \mathbb{R}_{+}$, by dominated convergence, we obtain
\begin{align*}
&\underset{m,d,n,k \to \infty}{\lim}  \int_{0}^{1} dr h(r) \int_{0}^{\infty} db \ p^{\delta}_{r}(b) \frac{1}{b^{3}} \mathcal{T}^{0}_{b} E^{\delta} [\Phi^{m, d}_{n,k}(X) | X_{r} = \cdot] \\
&=  \int_{0}^{1} dr h(r) \int_{0}^{\infty} db \ p^{\delta}_{r}(b) \frac{1}{b^{3}} \mathcal{T}^{0}_{b} E^{\delta} [\Phi(X) | X_{r} = \cdot].
\end{align*}
We have thus proved that, when we send $k$,$n$,$d$ and $m$ to $+\infty$ in \eqref{general_formula_13_mdnk}, all the terms converge to the same terms with $\Phi^{m, d}_{n,k}$ replaced with $\Phi$. We thus obtain the claim.
\end{proof}

\subsection{Extension of the IbPF for $\delta =1$}


\begin{proof}[Proof of Theorem \eqref{general_IbPF_1}]
Let $\Phi \in  \mathcal{S}C^{1}_{b}\left( L^{1}(0,1) \right)$. Consider $(\Phi^{d}_{n,k})_{d,n,k \geq 1}$ approximating $\Phi$ as in Proposition \ref{density_l1}. Then, for all $d,n,k \geq 1$, $\Phi^{d}_{n,k} \in \mathscr{S}$. Hence, by \eqref{one} combined with \eqref{reexpressed_term_1}, we have
\begin{align}
\label{general_formula_1_mdnk}
E^{1} (\partial_{h} \Phi^{d}_{n,k} ) =& - E^{1} (\langle h '' , X \rangle \Phi^{d}_{n,k} ) \\
\nonumber &- \frac{1}{2\sqrt{2 \pi}} \int_{0}^{1} dr \frac{h(r)}{(r(1-r))^{3/2}} E^{1} [\Phi^{d}_{n,k} | X_{r} = 0]  \\
\nonumber & + \frac{1}{2\sqrt{2 \pi}} \int_{0}^{1} dr \frac{h(r)}{(r(1-r))^{1/2}} \frac{d^{2}}{db^{2}} E^{1} [\Phi^{d}_{n,k} | X_{r} = b]  \biggr\rvert_{b=0}  .
\end{align}
Reasoning as in the proof of Theorem \ref{general_IbPF_13}, we obtain
\[ \underset{d,n,k \to \infty}{\lim} E^{1} (\partial_{h} \Phi^{d}_{n,k}(X) ) = E^{1} (\partial_{h} \Phi(X) ), \]
\[  \underset{d,n,k \to \infty}{\lim} E^{1} (\langle h '' , X \rangle\Phi^{d}_{n,k}(X) ) = E^{1} (\langle h '' , X \rangle\Phi(X) ), \]
and
\[ \underset{d,n,k \to \infty}{\lim}  \int_{0}^{1} dr \frac{h(r)}{(r(1-r))^{3/2}} E^{1} [\Phi^{d}_{n,k}(X) | X_{r} = 0]=  \int_{0}^{1} dr \frac{h(r)}{(r(1-r))^{3/2}} E^{1} [\Phi(X) | X_{r} = 0], \]
where we take the limits $k,n$ and $d$ successively. Hence, there only remains to study the last term in the right-hand side of \eqref{general_formula_1_mdnk}. For that term, note that, for all $d,n,k \geq 1$, $r \in (0,1)$ and $b \geq 0$, by Proposition \ref{diff_prop}, we have
\begin{align*}
&\frac{d^{2}}{db^{2}}  E^{1} [\Phi^{d}_{n,k} | X_{r} = b] \biggr\rvert_{b=0}  = \\
&2 \int \left(  E^{1} \left[ \Phi^{d}_{n,k} \left(\sqrt{ X^{2} + Z} \right) | X_{r} =0 \right] - E^{1} \left[\Phi^{d}_{n,k} | X_{r} =0 \right] \right) dM^{r}(Z).
\end{align*}
and, similarly
\begin{align*}
&\frac{d^{2}}{db^{2}}  E^{1} [\Phi(X) | X_{r} = b] \biggr\rvert_{b=0}  = \\
&2 \int \left(  E^{1} \left[ \Phi \left(\sqrt{ X^{2} + Z} \right) | X_{r} =0 \right] - E^{1} \left[ \Phi(X) | X_{r} =0 \right] \right) dM^{r}(Z).
\end{align*}
Now, as a consequence of condition \ref{fst_condition} in Definition \ref{pdi_conv}, by dominated convergence, for all $Z \in C_{+}([0,1])$ we have
\[ \begin{split}
\underset{d,n,k \to \infty}{\lim} &E^{1} \left[  \Phi^{d}_{n,k}\left( \sqrt{ X^{2} + Z} \right) | X_{r} =0 \right] - E^{1} \left[  \Phi^{d}_{n,k}(X) | X_{r} =0 \right] \\
&=  E^{1} \left[ \Phi \left(\sqrt{ X^{2} + Z} \right) | X_{r} =0 \right] - E^{1} \left[ \Phi(X) | X_{r} =0 \right],
\end{split}\]
and by condition \ref{thd_condition}, all three convergences happen with uniform domination by $\| Z \|_{1}$. Since, by Lemma \ref{finite_moment_nr}, we have
\[ \int_{0}^{1} dr \frac{|h(r)|}{(r(1-r))^{1/2}}\int \|Z\|_{1} M^{r}(dZ)  \leq \frac{1}{3}  \int_{0}^{1} dr \frac{|h(r)|}{(r(1-r))^{1/2}} < \infty, \]
by dominated convergence, we deduce that the quantity  
\begin{align*}
\int_{0}^{1} dr \frac{h(r)}{(r(1-r))^{1/2}} \int \left(  E^{1} \left[\Phi^{d}_{n,k} \left(\sqrt{ X^{2} + Z} \right) | X_{r} =0 \right] - E^{1} \left[ \Phi^{d}_{n,k}(X) | X_{r} =0 \right] \right) dM^{r}(Z)
\end{align*}
converges, as we send $k,n,d \to \infty$, to 
\begin{align*}
\int_{0}^{1} dr \frac{h(r)}{(r(1-r))^{1/2}} \int \left(  E^{1} \left[\Phi(X) \left(\sqrt{ X^{2} + Z} \right) | X_{r} =0 \right] - E^{1} \left[\Phi(X) | X_{r} =0 \right] \right) dM^{r}(Z).
\end{align*}
That is
\[\begin{split}
\underset{d,n,k \to \infty}{\lim} &\int_{0}^{1} dr \frac{h(r)}{(r(1-r))^{1/2}} \frac{d^{2}}{db^{2}} E^{1} [\Phi^{d}_{n,k} (X) | X_{r} = b]   \biggr\rvert_{b=0}  \\
= &\int_{0}^{1} dr \frac{h(r)}{(r(1-r))^{1/2}} \frac{d^{2}}{db^{2}}  E^{1} [\Phi(X) | X_{r} = b] \biggr\rvert_{b=0}. \end{split}\]
We thus obtain the claim.
\end{proof}

\subsection{IbPF for $ \delta \in (0,1)$}


\begin{proof}
Let $\Phi \in \mathcal{S}C^{1,1}_{b}\left( L^{1}(0,1) \right)$, and consider $(\Phi^{d}_{n,k})_{d,n,k \geq 1}$ approximating $\Phi$ as in Proposition \ref{density_l1_lip}. Then, for all $d,n,k \geq 1$, $\Phi^{d}_{n,k} \in \mathscr{S}$. Hence, by \eqref{onetothree} combined with \eqref{reexpressed_term_01}, we have
\begin{align}
\label{general_formula_01_dnk}
& E^{\delta} \left[ \partial_{h} \Phi^{d}_{n,k} (X) \right] = - E^{\delta} \left[ \langle h '' , X \rangle \Phi^{d}_{n,k} (X) \right] \\
\nonumber & - \kappa(\delta) \int_{0}^{1} dr \ h(r)  \int_{0}^{\infty} db \ \frac{p^{\delta}_{r}(b)}{b^{3}} \mathcal{T}^{2}_{0,b} E^{\delta} [\Phi^{d}_{n,k}(X) | X_{r} = \cdot]  \\
\nonumber & - \frac{\kappa(\delta) \Gamma(\frac{\delta-3}{2})}{\Gamma(\frac{\delta}{2})} \int_{0}^{1} dr \frac{h(r)}{(2r(1-r))^{3/2}}  E^{\delta} [\Phi^{d}_{n,k}(X) | X_{r} = 0] \\
\nonumber & - \frac{\kappa(\delta) \Gamma(\frac{\delta-1}{2})}{2 \Gamma(\frac{\delta}{2})} \int_{0}^{1} dr \frac{h(r)}{(2r(1-r))^{1/2}} \frac{d^{2}}{dx^{2}} E^{\delta} [\Phi^{d}_{n,k}(X) | X_{r} = x] \biggr\rvert_{x=0}. 
\end{align}
Reasoning exactly as in the proofs of Theorems \ref{general_IbPF_13} and \ref{general_IbPF_1}, we obtain that
\[ \underset{d,n,k \to \infty}{\lim} E^{\delta} (\partial_{h} \Phi^{d}_{n,k}(X) ) = E^{\delta} (\partial_{h} \Phi(X) ), \]
\[  \underset{d,n,k \to \infty}{\lim} E^{\delta} (\langle h '' , X \rangle\Phi^{d}_{n,k}(X) ) = E^{\delta} (\langle h '' , X \rangle\Phi(X) ), \]
\[ \underset{d,n,k \to \infty}{\lim}  \int_{0}^{1} dr \frac{h(r)}{(r(1-r))^{3/2}} E^{\delta} [\Phi^{d}_{n,k}(X) | X_{r} = 0]=  \int_{0}^{1} dr \frac{h(r)}{(r(1-r))^{3/2}} E^{\delta} [\Phi(X) | X_{r} = 0], \]
and
\[\begin{split}
\underset{d,n,k \to \infty}{\lim} &\int_{0}^{1} dr \frac{h(r)}{(r(1-r))^{1/2}} \frac{d^{2}}{db^{2}} E^{\delta} [\Phi^{d}_{n,k} (X) | X_{r} = b]   \biggr\rvert_{b=0}  \\
= &\int_{0}^{1} dr \frac{h(r)}{(r(1-r))^{1/2}} \frac{d^{2}}{db^{2}}  E^{\delta} [\Phi(X) | X_{r} = b] \biggr\rvert_{b=0}. 
\end{split}\]
Hence, there only remains to study the second term in the right-hand side of \eqref{general_formula_01_dnk}. Reasoning as before, we see that, for all $r \in (0,1)$, $b \geq 0$, we have
\[\underset{d,n,k \to \infty}{\lim} \mathcal{T}^{2}_{0,b} E^{\delta} [\Phi^{d}_{n,k}(X) | X_{r} = \cdot] = \mathcal{T}^{2}_{0,b} E^{\delta} [\Phi^{d}(X) | X_{r} = \cdot],\]
and, since the approximating family $(\Phi^{d}_{n,k})_{d,n,k \geq 1}$ satisfies the domination assumption \eqref{condition_pdi_lip}, by Proposition \ref{taylor_est_cond_exp}, all three limits happen with uniform domination by $b^{4}$. Since 
\[ \int_{0}^{1} dr \ |h(r)|  \int_{0}^{\infty} db \ \frac{p^{\delta}_{r}(b)}{b^{3}} b^{4} < \infty, \]
 by dominated convergence, we deduce that 
\begin{align*}
&\underset{d,n,k \to \infty}{\lim} \int_{0}^{1} dr \ h(r)  \int_{0}^{\infty} db \ \frac{p^{\delta}_{r}(b)}{b^{3}} \mathcal{T}^{2k}_{0,b} E^{\delta} [\Phi^{d}_{n,k}(X) | X_{r} = \cdot] \\
&= \int_{0}^{1} dr \ h(r)  \int_{0}^{\infty} db \ \frac{p^{\delta}_{r}(b)}{b^{3}} \mathcal{T}^{2k}_{0,b} E^{\delta} [\Phi^{d}(X) | X_{r} = \cdot].
\end{align*}
We thus obtain the claim.
\end{proof}

\begin{rk}[An open question]
As mentioned in Remark \ref{open_quest_estimates}, it is still unknown whether Theorems \ref{general_IbPF_1} and \ref{general_IbPF_01} apply for all $\Phi \in C^{1}_{b}(L^{2}(0,1))$. Answering this question would require to obtain either sharpness statements or refinements of the estimates obtained in Section \ref{section_estimates}.
\end{rk}

\section{Appendix}
\label{sect_appendix}

We now give a proof of the approximation Lemmas \ref{pseudo_Weierstrass} and \ref{pseudo_Weierstrass_bis} we used in Section \ref{section_density} to approximate sufficiently regular functions on $\mathbb{R}_{+}^{d}$ by linear combinations of exponential functions. The main idea is simply to proceed to a change of variable using the exponential function, so that we are led to the problem of approximating functions on $[0,1]^{d}$ by polynomials; this, in turn, is done using Bernstein polynomials. Note that while Lemma \ref{pseudo_Weierstrass} is a consequence of Theorem 1.1.2 in \cite{llavona1986approximation}, we could not find in the literature a version of the Weierstrass approximation Theorem  yielding the particular type of convergence needed in Lemma \ref{pseudo_Weierstrass_bis}. We therefore propose a construction of the approximating sequences which works for both lemmas.
  
\subsection{Proof of Lemma \ref{pseudo_Weierstrass}} 
 
\begin{proof}
Define $f:[0,1]^{d} \to \mathbb{R}$ by setting
\[f(y_{1}, \ldots,y_{d}) = h \left( - \ln(y_{1}), \ldots, - \ln(y_{d}) \right), \qquad (y_1, \ldots, y_d) \in  (0,1]^{d},\]
and $f(y_{1}, \ldots,y_{d}) = 0$ if $y_i=0$ for some $i$. For all $k \geq 0$, define the polynomial function $P_{k}f$ on $[0,1]^{d}$ by
\[  P_{k}f(y) := \sum_{\substack{\ell = (\ell_1, \ldots, \ell_d)\\ 0 \leq \ell_{1}, \ldots, \ell_{d} \leq k}} f\left(\frac{ \ell}{k} \right) \prod_{i=1}^{d} B^{k}_{\ell_{i}}(y_{i}), \quad y \in [0,1]^{d}, \] 
where we use the notation $\frac{\ell}{k} := (\frac{\ell_1}{k}, \ldots, \frac{\ell_d}{k})$ and, for all $0\leq m \leq k$, $B^{k}_{m}$ is the Bernstein polynomial defined by
\[ B^{k}_{m}(X) := {k \choose m} X^{m}(1-X)^{k-m}. \]
Note that these polynomials form a partition of unity:
\begin{equation}
\label{part_of_unity}
\forall k \geq 0, \quad \sum_{m=0}^{k} B^{k}_{m}(X) = 1.
\end{equation}
We claim that the following holds:
\begin{itemize}
\item for all $y \in [0,1]^{d}$, $P_{k}f(y) \underset{k \to \infty}{\longrightarrow} f(y)$, and
\[ \forall k \geq 0, \quad \|P_{k}f\|_{\infty} \leq \|f\|_{\infty}\]
\item for all $y \in [0,1]^{d}$, $\nabla P_{k}f(y)   \underset{k \to \infty}{\longrightarrow} \nabla f(y)$, and
\[\forall k \geq 0, \forall i = 1, \ldots, d, \quad \|\partial_{i} P_{k}f\|_{\infty} \leq \| \partial_{i} f \|_{\infty}.\]
\end{itemize}
To prove the first point, note that, for all $y \in [0,1]^{d}$, we have
\begin{equation}
\label{prob_repres}
 P_{k}f(y) = \mathbb{E} \left[ f \left( \frac{S^{1}_{k}}{k}, \ldots, \frac{S^{d}_{k}}{k} \right) \right]
\end{equation}
where, for $1 \leq i \leq d$, we set $S^{i}_{k} := \sum_{j=1}^{k} X^{i}_{j}$, with $X^{i}_{j}$ a Bernoulli variable of parameter $y_{i}$,
the family of random variables $(X^{i}_{j})_{\substack{1 \leq i\leq d \\ 1 \leq j \leq k}}$ being independent. By the weak law of large numbers, for all $i= 1, \ldots, d$, $\frac{S^{i}_{k}}{k} \underset{k \to \infty}{\longrightarrow} y_{i}$ in probability.
 Hence, we have the convergence in probability
\[   \left( \frac{S^{1}_{k}}{k}, \ldots, \frac{S^{d}_{k}}{k} \right)  \underset{k \to \infty}{\longrightarrow} (y_{1}, \ldots, y_{d}), \]
so that, since $f$ is bounded and continuous on $[0,1]^{d}$, we deduce that $P_{k}f(y) \underset{k \to \infty}{\longrightarrow} f(y)$. Moreover, from the representation \eqref{prob_repres}, we see that $\|P_{k}f\|_{\infty} \leq \|f\|_{\infty}$. We now establish the second point. For all $i= 1, \ldots, d$ and $y \in [0,1]^{d}$, we have
\[ \partial_{i} P_{k}f (y) =  \sum_{0 \leq \ell_{1}, \ldots, \ell_{d} \leq k } f \left( \frac{\ell}{k} \right) B^{k}_{\ell_{1}}(y_{1}) \ldots {B^{k}_{\ell_{i}}}'(y_{i }) \ldots B^{k}_{\ell_{d}}(y_{d}). \] 
But, for all $m = 1, \ldots,  k$, it holds ${B^{k}_m}' = k \left( B^{k-1}_{m-1} - B^{k-1}_{m} \right)$ (with the convention $B^{n}_{m} =0$ if $n < 0$, $m < 0$ or $m > n$).  Therefore, we have
\[\partial_{i} P_{k}f (y) =  \sum_{0 \leq \ell_{1}, \ldots, \ell_{d} \leq k } f \left( \frac{\ell}{k}\right) k \left( B^{k-1}_{\ell_{i}-1}(y_{i}) - B^{k-1}_{\ell_{i}}(y_{i }) \right)  \prod_{j \neq i} B^{k}_{\ell_{j}}(y_{j}), \]
which, after a discrete summation by parts, yields
\[\partial_{i} P_{k}f (y) =  \sum_{0 \leq \ell_{1}, \ldots, \ell_{d} \leq k } k \left( f \left( \frac{\ell}{k} + \frac{1}{k} e_i \right) -  f \left( \frac{\ell}{k} \right) \right) B^{k-1}_{\ell_{i}}(y_{i })  \prod_{j \neq i} B^{k}_{\ell_{j}}(y_{j}), \]
where we have denoted by $(e_1, \ldots,e_d)$ the canonical basis of $\mathbb{R}^d$. 
Now, for all $\ell \in \{0, \ldots,k\}^{d}$, we have
\[ \left| f \left( \frac{\ell}{k} + \frac{1}{k} e_i \right) - f \left( \frac{\ell}{k} \right) \right| \leq \frac{1}{k} \| \partial_{i} f \|_{\infty}, \]
so that, recalling \eqref{part_of_unity}, we obtain
\begin{align*}
| \partial_{i} P_{k}f (y) | &\leq \| \partial_{i} f \|_{\infty} \sum_{0 \leq \ell_{1}, \ldots, \ell_{d} \leq k } B^{k-1}_{\ell_{i}}(y_{i })  \prod_{j \neq i} B^{k}_{\ell_{j}}(y_{j}) \\
&= \| \partial_{i} f \|_{\infty}. 
\end{align*}
Moreover, we can write
\begin{align}
\label{deriv_polynom}
\partial_{i} P_{k}f (y) &= \sum_{0 \leq \ell_{1}, \ldots, \ell_{d} \leq k } \partial_{i} f \left( \frac{\ell}{k} \right)  B^{k-1}_{\ell_{i}}(y_{i })  \prod_{j \neq i} B^{k}_{\ell_{j}}(y_{j}) \\
\nonumber &+  \sum_{0 \leq \ell_{1}, \ldots, \ell_{d} \leq k } R(k,\ell) B^{k-1}_{\ell_{i}}(y_{i })  \prod_{j \neq i} B^{k}_{\ell_{j}}(y_{j}),
\end{align}
where, for all $\ell \in \{0, \ldots,k\}^{d}$,
\[ R(k,\ell) :=  k \left( f \left( \frac{\ell}{k} + \frac{1}{k} e_i \right) -  f \left( \frac{\ell}{k} \right) \right) -  \partial_{i} f \left( \frac{\ell}{k} \right). \]
Since $\partial_{i} f$ is continuous on $[0,1]^{d}$, reasoning as for $f$, we obtain that the first term in the RHS of \eqref{deriv_polynom} converges, as $k \to \infty$, to $\partial_{i} f(y)$. Regarding the second term, note that 
\begin{align*}
|R(k,\ell)| &= k \left| \int_{\frac{\ell_{i}}{k}}^{\frac{\ell_{i}+1}{k}} \left(\partial_{i} f \left( \frac{\ell_{1}}{k},\ldots, t, \ldots, \frac{\ell_{d}}{k} \right) - \partial_{i} f \left( \frac{\ell_{1}}{k},\ldots, \frac{\ell_{i}}{k}, \ldots, \frac{\ell_{d}}{k} \right) \right) dt \right| \\
&\leq k \int_{\frac{\ell_{i}}{k}}^{\frac{\ell_{i}+1}{k}} \omega \left( \partial_{i}f, \frac{1}{k} \right)dt =    \omega \left( \partial_{i}f, \frac{1}{k} \right),
\end{align*}
where $ \omega \left( \partial_{i}f, \cdot \right)$ denotes the modulus of continuity of $\partial_{i}f$ on $[0,1]^{d}$. Therefore, the second term in the right-hand side of \eqref{deriv_polynom} is dominated by
\[\omega \left( \partial_{i}f, \frac{1}{k} \right) \sum_{0 \leq \ell_{1}, \ldots, \ell_{d} \leq k } B^{k-1}_{\ell_{i}}(y_{i })  \prod_{j \neq i} B^{k}_{\ell_{j}}(y_{j}) = \omega \left( \partial_{i}f, \frac{1}{k} \right), \]
which converges to $0$ as $k \to \infty$. Therefore, sending $k \to \infty$ in \eqref{deriv_polynom}, we deduce that 
\[ \partial_{i} P_{k}f(y) \underset{k \to \infty}{\longrightarrow} \partial_{i} f(y).\]
This proves the second point. \\
We can now conclude the proof of the lemma. Indeed, setting, for all $k \in \mathbb{N}$ and $x \in \mathbb{R}_{+}^{d}$,
\[ h_{k}(x) :=  P_{k}f \left( e^{-x_{1}}, \ldots, e^{-x_{d}} \right), \]
it follows that $h_{k}$ has the requested form. Moreover, by the first point above, we have
 \[ \forall x \in \mathbb{R}_{+}^{d}, \quad h_{k}(x) \underset{k \to \infty}{\longrightarrow} f\left( e^{-x_{1}}, \ldots, e^{-x_{d}} \right) = h(x_{1}, \ldots, x_{d}), \]
together with the domination
\[ \forall k \in \mathbb{N}, \quad \|h_{k}\|_{\infty} \leq \|P_{k}f\|_{\infty} \leq \|f\|_{\infty}. \]
On the other hand, by the second point above, for all $i =1, \ldots, d$ and $x \in \mathbb{R}_{+}^{d}$, we have
\begin{align*} 
\partial_{i} h_{k}(x) =& - e^{-x_{i}} \partial_{i} P_{k}f \left( e^{-x_{1}}, \ldots, e^{-x_{d}} \right) \\ 
\underset{k \to \infty}{\longrightarrow}& - e^{-x_{i}} \partial_{i} f \left( e^{-x_{1}}, \ldots, e^{-x_{d}} \right) = \partial_{i} h(x).
\end{align*} 
Moreover, we have
\[ \forall k \in \mathbb{N}, \quad  \| \partial_{i} h_{k} \|_{\infty} \leq \|\partial_{i}f\|_{\infty}. \]
But, for all $ y \in [0,1]^{d}$,
\[ \partial_{i}f(y) = \frac{1}{y_{i}} \partial_{i} h \left( - \ln(y_{1}), \ldots, - \ln(y_{d}) \right),\] 
and, since $\partial_{i} h$ is supported in $[0,n]^{d}$, so that $\partial_{i} f$ is supported in $[e^{-n},1]^{d}$, we get
\begin{align*}
\|\partial_{i}f\|_{\infty}  = \sup_{y \in [e^{-n},1]^{d}} \left|  \frac{1}{y_{i}} \partial_{i} h \left( - \ln(y_{1}), \ldots, - \ln(y_{d}) \right) \right| \leq e^{n} \|\partial_{i} h \|_{\infty},
\end{align*}
whence
\[ \forall k \in \mathbb{N}, \quad  \| \partial_{i} h_{k} \|_{\infty} \leq e^{n} \|\partial_{i} h \|_{\infty}, \]
which gives the requested bound (with $C(n) := e^{n}$) . The lemma is proved.
\end{proof}

\subsection{Proof of Lemma \ref{pseudo_Weierstrass_bis}}

\begin{proof}
To obtain the claim, it suffices to show that, as a consequence of the estimate \eqref{lip_deriv_g}, the sequence of functions $(h_{k})_{k \geq 0}$ constructed in the proof of Lemma \ref{pseudo_Weierstrass} satisfies, for all $k \geq 0$ and $i = 1, \ldots, d$,
\[ \forall x, y  \in \mathbb{R}_{+}^{d}, \quad | \partial_{i} h _{k} (x) - \partial_{i} h _{k} (y) | \leq C'(n) \left( L' + \| \partial_{i} h \|_{\infty}  \right) \sum_{j=1}^{d} |x_{j} - y_{j}| \]
for some constant $C'(n) > 0$. From now on, let $k \geq 0$ and $i = 1, \ldots, d$ be fixed. First note that, for all $u,v \in [e^{-n-1}, 1]^{d}$, we have
\begin{align*}
| \partial_{i} f (u) -  \partial_{i} f (u) | &= \left| \frac{1}{u_{i}}\partial_{i} h \left( - \ln(u_{1}), \ldots, - \ln(u_{d}) \right)  - \frac{1}{v_{i}} \partial_{i} h \left( - \ln(v_{1}), \ldots, - \ln(v_{d}) \right) \right| \\
&\leq   \frac{1}{u_{i}} \left|  \partial_{i} h \left( - \ln(u_{1}), \ldots, - \ln(u_{d}) \right)  -  \partial_{i} h \left( - \ln(v_{1}), \ldots, - \ln(v_{d}) \right) \right| \\
&+ \left|  \frac{1}{u_{i}}  -  \frac{1}{v_{i}} \right| \left| \partial_{i} h \left( - \ln(v_{1}), \ldots, - \ln(v_{d}) \right) \right| \\
&\leq e^{n+1} L' \sum_{j=1}^{d} |\ln(u_{j}) - \ln(v_{j}) |  + \frac{|u_{i} - v_{i}|}{u_{i}v_{i}} \| \partial_{i} h \|_{\infty} \\
&\leq e^{2(n+1)} L' \sum_{j=1}^{d} |u_{j} - v_{j}| + e^{2(n+1)}   \| \partial_{i} h \|_{\infty} |u_{i} - v_{i}| \\
&\leq e^{2(n+1)} \left( L' + \| \partial_{i} h \|_{\infty}  \right) \sum_{j=1}^{d} |u_{j} - v_{j}| .
\end{align*}
Moreover, since $f$ is supported in $ [e^{-n}, 1]^{d}$, for all $u,v \notin [e^{-n}, 1]^{d}$, we have
\[ | \partial_{i} f (u) -  \partial_{i} f (u) | =0 \]
Finally, for all $u \in [e^{-n}, 1]^{d}$ and $v \notin [e^{-n-1}, 1]^{d}$, we have
\begin{align*}
| \partial_{i} f (u) -  \partial_{i} f (u) | &= | \partial_{i} f (u)  | \leq \| \partial_{i} f \|_{\infty} \leq e^{n} \| \partial_{i} h \|_{\infty}, 
\end{align*}
and, since $\sum_{j=1}^{d} |u_{j} - v_{j}| \geq e^{-n}(1-e^{-1}) \geq e^{-n- 2}$ by our assumption on $u$ and $v$, we deduce that
\begin{align*}
| \partial_{i} f (u) -  \partial_{i} f (u) | \leq  e^{2n+2} \| \partial_{i} h \|_{\infty} \sum_{j=1}^{d} |u_{j} - v_{j}| , 
\end{align*}
and by symmetry the same bound holds when $u \notin [e^{-n-1}, 1]^{d}$ and $v \in [e^{-n}, 1]^{d}$. Thus, we deduce that, for all $u,v \in [0,1]^{d}$, we have
\begin{equation}
\label{estim_deriv_f}
 | \partial_{i} f (u) -  \partial_{i} f (u) | \leq e^{2(n+1)} \left( L' + \| \partial_{i} h \|_{\infty}  \right) \sum_{j=1}^{d} |u_{j} - v_{j}|.
\end{equation} 
We will use this estimate to bound the second-order partial derivatives of $P_{k}f$. Let first $ j \in \{1, \ldots, d \}$ such that $j \neq i$, and suppose, for example, that $j > i$. Then recall from the proof of Lemma \ref{pseudo_Weierstrass} that, for all $u \in [0,1]^{d}$,
\[ \partial_{i} P_{k}f (u) =  \sum_{0 \leq \ell_{1}, \ldots, \ell_{d} \leq k } k \left( f \left(  \frac{\ell}{k} + \frac{1}{k} e_i\right) -  f \left( \frac{\ell}{k} \right) \right) B^{k-1}_{\ell_{i}}(u_{i })  \prod_{j \neq i} B^{k}_{\ell_{j}}(u_{j}). \]
By the same computations, we get
\[ \partial^{2}_{i,j} P_{k}f (u) =  \sum_{0 \leq \ell_{1}, \ldots, \ell_{d} \leq k } D(k,\ell) B^{k-1}_{\ell_{i}}(u_{i }) B^{k-1}_{\ell_{j}}(u_{j })  \prod_{m \neq i,j} B^{k}_{\ell_{m}}(u_{m}), \]
where, for all $\ell \in \{0, \ldots, k \}^{d}$,
\begin{align*}
D(k, \ell) := k^{2} \Bigg[& f \left(  \frac{\ell}{k} + \frac{1}{k} e_i +\frac{1}{k} e_j \right) -  f \left(  \frac{\ell}{k} + \frac{1}{k} e_i\right) \\
&- f \left(\frac{\ell}{k} + \frac{1}{k} e_j \right) +  f \left( \frac{\ell}{k} \right) \Bigg]. 
\end{align*}
Hence, as a consequence of \eqref{estim_deriv_f}, we have
\begin{align*}
|D(k, \ell)| &= k^{2} \left| \int_{0}^{1/k} \left( \partial_{j}  f \left(  \frac{\ell}{k} + \frac{1}{k} e_{i} + t e_{j} \right) -  \partial_{j}f \left(  \frac{\ell}{k} + t e_{j} \right) \right) dt \right| \\ 
&\leq k^{2} e^{2(n+1)} \int_{0}^{1/k} \left( L' + \| \partial_{i} h \|_{\infty}  \right) \frac{1}{k} dt \leq e^{2(n+1)} \left( L' + \| \partial_{i} h \|_{\infty}  \right) . 
\end{align*}
Therefore, for all $u \in [0,1]^{d}$, recalling \eqref{part_of_unity}, we have
\begin{align*}
|\partial^{2}_{i,j} P_{k}f (u)| &\leq  e^{2(n+1)} \left( L' + \| \partial_{i} h \|_{\infty}  \right)  \sum_{0 \leq \ell_{1}, \ldots, \ell_{d} \leq k } B^{k-1}_{\ell_{i}}(u_{i }) B^{k-1}_{\ell_{j}}(u_{i })  \prod_{m \neq i,j} B^{k}_{\ell_{m}}(u_{m}) \\
&=  e^{2(n+1)} \left( L' + \| \partial_{i} h \|_{\infty}  \right) . 
\end{align*}
In a similar way we obtain that, for all $u \in [0,1]^{d}$, 
 \[ |\partial^{2}_{i,i} P_{k}f (u)| \leq e^{2(n+1)} \left( L' + \| \partial_{i} h \|_{\infty}  \right) . \]
Recall now that $h_{k}$ is defined, for all $x \in \mathbb{R}_{+}^{d}$, by $h_{k}(x) = P_{k}f \left( e^{-x_{1}}, \ldots, e^{-x_{d}} \right)$. Hence, for all $j \neq i$ and $x \in \mathbb{R}_{+}^{d}$, we have
\begin{align*}
|\partial^{2}_{i,j} h_{k} (x)| &= |e^{-x_{i}} e^{-x_{j}} \partial^{2}_{i,j} P_{k}f  \left( e^{-x_{1}}, \ldots, e^{-x_{d}} \right) | \\
&\leq e^{2(n+1)} \left( L' + \| \partial_{i} h \|_{\infty}  \right).
\end{align*}
On the other hand, for all $x \in \mathbb{R}_{+}^{d}$,
\begin{align*}
\partial^{2}_{i,i} h_{k} (x) = e^{-x_{i}} \partial_{i} P_{k}f  \left( e^{-x_{1}}, \ldots, e^{-x_{d}} \right) + e^{-2 x_{i}} \partial^{2}_{i,i} P_{k}f  \left( e^{-x_{1}}, \ldots, e^{-x_{d}} \right), 
\end{align*}
so that
\begin{align*}
|\partial^{2}_{i,i} h_{k} (x)| &\leq \| \partial_{i} P_{k}f \|_{\infty} + \| \partial^{2}_{i,i} P_{k}f \|_{\infty} \\
&\leq e^{n}  \| \partial_{i} h \|_{\infty} + e^{2(n+1)} \left( L' + \| \partial_{i} h \|_{\infty}  \right) \\
&\leq 2 e^{2(n+1)} \left( L' + \| \partial_{i} h \|_{\infty}  \right).
\end{align*}
We have thus proved that, for all $j =1, \ldots, d$ and $x \in \mathbb{R}_{+}^{d}$,
\[ |\partial^{2}_{i,j} h_{k} (x)| \leq 2 e^{2(n+1)} \left( L' + \| \partial_{i} h \|_{\infty}  \right).\]
Therefore, for all $x, y \in \mathbb{R}^{+}_{d}$, we have
\[ | \partial_{i} h_{k} (x) - \partial_{i} h_{k} (y) | \leq 2 e^{2(n+1)} \left( L' + \| \partial_{i} h \|_{\infty}  \right) \sum_{j=1}^{d} |x_{j} - y_{j}|, \]
which yields the desired bound.
\end{proof}

\bibliographystyle{amsplain}
\bibliography{Taylor_Bessel_and_IbPF}

\providecommand{\bysame}{\leavevmode\hbox to3em{\hrulefill}\thinspace}
\providecommand{\MR}{\relax\ifhmode\unskip\space\fi MR }
\providecommand{\MRhref}[2]{%
  \href{http://www.ams.org/mathscinet-getitem?mr=#1}{#2}
}
\providecommand{\href}[2]{#2}
\begin{thebibliography}{10}

\bibitem{eladaltman2019bessel}
Henri~Elad Altman, \emph{Bessel {SPDE}s with general {D}irichlet boundary
  conditions}, Electronic Journal of Probability \textbf{26} (2021), no.~76,
  1--36.

\bibitem{ASZ}
L.~Ambrosio, G.~Savar\'{e}, and L.~Zambotti, \emph{Existence and stability for
  {F}okker-{P}lanck equations with log-concave reference measure}, Probab.
  Theory Related Fields \textbf{145} (2009), no.~3-4, 517--564. \MR{2529438}

\bibitem{eladaltman2020}
H.~Elad~Altman, \emph{Integration by parts formulae for the laws of {B}essel
  bridges via hypergeometric functions}, Electron. Commun. Probab. \textbf{25}
  (2020), 11 pp.

\bibitem{EladAltman2019}
H.~{Elad Altman} and L.~Zambotti, \emph{{Bessel SPDEs and renormalised local
  times}}, Probability Theory and Related Fields (2019).

\bibitem{fukushima2010dirichlet}
M.~Fukushima, Y.~Oshima, and M.~Takeda, \emph{Dirichlet forms and symmetric
  {M}arkov processes}, vol.~19, Walter de Gruyter, 2010.

\bibitem{funaki2001fluctuations}
T.~Funaki and S.~Olla, \emph{Fluctuations for $\nabla \varphi$ interface model
  on a wall}, Stochastic processes and their applications \textbf{94} (2001),
  no.~1, 1--27.

\bibitem{grothaus2016integration}
M.~Grothaus and R.~Vo{\ss}hall, \emph{Integration by parts on the law of the
  modulus of the {B}rownian bridge}, arXiv preprint arXiv:1609.02438 (2016).

\bibitem{ikeda2014stochastic}
N.~Ikeda and S.~Watanabe, \emph{Stochastic differential equations and diffusion
  processes}, vol.~24, Elsevier, 2014.

\bibitem{llavona1986approximation}
J.~G. Llavona, \emph{Approximation of continuously differentiable functions},
  vol. 130, Elsevier, 1986.

\bibitem{nualart1992white}
D.~Nualart and E.~Pardoux, \emph{White noise driven quasilinear {SPDE}s with
  reflection}, Probability Theory and Related Fields \textbf{93} (1992), no.~1,
  77--89.

\bibitem{pitman1982decomposition}
J.~Pitman and M.~Yor, \emph{A decomposition of {B}essel bridges}, Zeitschrift
  f{\"u}r Wahrscheinlichkeitstheorie und verwandte Gebiete \textbf{59} (1982),
  no.~4, 425--457.

\bibitem{revuz2013continuous}
D.~Revuz and M.~Yor, \emph{Continuous martingales and {B}rownian motion}, vol.
  293, Springer Science \& Business Media, 2013.

\bibitem{shiga1973bessel}
T.~Shiga and S.~Watanabe, \emph{Bessel diffusions as a one-parameter family of
  diffusion processes}, Zeitschrift f{\"u}r Wahrscheinlichkeitstheorie und
  verwandte Gebiete \textbf{27} (1973), no.~1, 37--46.

\bibitem{zambotti2002integration}
L.~Zambotti, \emph{Integration by parts formulae on convex sets of paths and
  applications to {SPDE}s with reflection}, Probability Theory and Related
  Fields \textbf{123} (2002), no.~4, 579--600.

\bibitem{zambotti2003integration}
\bysame, \emph{Integration by parts on $\delta$-{B}essel bridges, $\delta> 3$,
  and related {SPDE}s}, The Annals of Probability \textbf{31} (2003), no.~1,
  323--348.

\bibitem{zambotti2004fluctuations}
\bysame, \emph{Fluctuations for a $\nabla \varphi$ interface model with
  repulsion from a wall}, Probability theory and related fields \textbf{129}
  (2004), no.~3, 315--339.

\bibitem{zambotti2005integration}
\bysame, \emph{Integration by parts on the law of the reflecting {B}rownian
  motion}, Journal of Functional Analysis \textbf{223} (2005), no.~1, 147--178.

\bibitem{zambotti2017random}
\bysame, \emph{Random obstacle problems: {\'E}cole d'{\'e}t{\'e} de
  probabilit{\'e}s de {S}aint-{F}lour xlv-2015}, vol. 2181, Springer, 2017.

\end{thebibliography}

\end{document}